%% file: main_arxiv.tex
\documentclass[11pt, letterpaper]{article}

\usepackage[
letterpaper,
top=1in,
bottom=1in,
left=1in,
right=1in]{geometry}

\usepackage[utf8]{inputenc}

\input{preamble.tex}

\usepackage{amsmath, amsthm, amssymb}
\usepackage[pagebackref,colorlinks=true,urlcolor=blue,linkcolor=blue,citecolor=OliveGreen]{hyperref}
\usepackage[nameinlink]{cleveref}

\let\subparagraph\paragraph

\usepackage{thmtools}
\usepackage{thm-restate}
\newtheorem{theorem}{Theorem}[section]

\newtheorem{lemma}[theorem]{Lemma}

\newtheorem{proposition}[theorem]{Proposition}

\theoremstyle{definition}
\newtheorem{definition}[theorem]{Definition}
\newtheorem{problem}[theorem]{Problem}
\newtheorem{example}[theorem]{Example}
\newtheorem{question}[theorem]{Question}
\newtheorem{remark}[theorem]{Remark}

\newtheorem{construction}[theorem]{Construction}

\hypersetup{
    pdfauthor = {Pepijn Roos Hoefgeest, Lucas Slot},
}

\title{Robustness of Persistent Topological Features \\ and Minimum Homological Cuts}

\author{
  \begin{minipage}[t]{0.45\textwidth}
    \centering
    Pepijn Roos Hoefgeest\thanks{This work was partially supported by the Wallenberg AI, Autonomous Systems and
Software Program (WASP) funded by the Knut and Alice Wallenberg Foundation.} \\ \small \href{mailto:pepijnrh@kth.se}{pepijnrh@kth.se} \\ \small KTH Stockholm
  \end{minipage}%
  \hfill
  \begin{minipage}[t]{0.45\textwidth}
    \centering
    Lucas Slot\thanks{This work was completed while the second author was at ETH Zurich, and supported by the Swiss National Science Foundation (SNSF), grant no. 10004947.} \\  \small \href{mailto:l.f.h.slot@uva.nl}{l.f.h.slot@uva.nl} \\ \small University of Amsterdam
  \end{minipage}%
}

\begin{document}

\maketitle
\thispagestyle{empty}
\begin{abstract}
\input{abstract}
\end{abstract}
\newpage

\thispagestyle{empty}
\tableofcontents
\newpage

\setcounter{page}{1}

\section{Introduction}
\input{Introduction}

\section{Preliminaries}
\input{Preliminaries.tex}

\section{Induced matchings, adversarial robustness and homological cuts}
\input{inducedmatchings.tex}

\section{Complexity of the minimum homological cut problem}
\input{zero_dimensional.tex}
\input{np-hardness.tex}
\input{alexander.tex}

\section{Adversarial robustness for the Rips filtration}
\input{HausdorffH.tex}

\section{Linear programming relaxations and homological max-flow min-cut}
\input{LP.tex}

\bibliographystyle{plain}
\bibliography{RTFbib}

\appendix
\crefalias{section}{appendix}
\crefalias{subsection}{appendix}

\section{Omitted proofs}
\input{Appendix.tex}

\end{document}

%% file: preamble.tex
\usepackage{amsthm, amsmath, amssymb}
\usepackage{mathtools}

\usepackage[dvipsnames]{xcolor}

\usepackage{tikz}
\usepackage{tikz-cd}
\usetikzlibrary{shapes.geometric}
\usetikzlibrary{arrows}
\usetikzlibrary{arrows.meta}
\usetikzlibrary{backgrounds}
\usetikzlibrary{decorations.markings}
\tikzset{    
    mypoint/.style={
        circle,
        draw,
        inner sep=.3mm
        },  
    whitepoint/.style={
        fill=white, 
        mypoint
        },  
    blackpoint/.style={
        fill=black, 
        mypoint
        },  
    textnode/.style={
        text height=2.5ex, 
        text depth=1ex
        },  
    }
    \newcommand{\midarrow}{\tikz \draw[-triangle 90] (0,0) -- +(.1,0);}

\newcommand{\R}{\mathbb{R}}
\newcommand{\N}{\mathbb{N}}
\newcommand{\Z}{\mathbb{Z}}
\newcommand{\mr}{\mathbf{r}}

\newcommand{\bc}{\mathcal{B}}
\renewcommand{\epsilon}{\varepsilon}
\newcommand{\im}{\mathcal{X}}

\newcommand{\PH}{\mathrm{PH}}

\newcommand{\simp}{K}
\newcommand{\fsimp}{\mathcal{K}}

\newcommand{\dualg}[1]{\mathcal{G}_{#1}}
\newcommand{\vcomp}[1]{V_{\mathbb{R}^n\setminus #1}}

\newcommand{\Ima}{\mathrm{Im}}
\newcommand{\incl}{\iota}
\newcommand{\ihomo}{\iota^*}
\newcommand{\imatch}{\mathcal{X}}

\newcommand{\VR}{\mathcal{R}}
\newcommand{\OVR}{\mathcal{R}^{\preceq}}

\newcommand{\imord}{\mathcal{X}^{\preceq}}

\newcommand{\predeath}[1]{K_{#1}}
\newcommand{\deathsimplex}[1]{\tau_{#1}}

\newcommand{\br}{\mathbf{r}}

\newcommand{\hh}[1]{\mathcal{H}_{#1}}
\newcommand{\hhy}{\hh{X,k}}
\newcommand{\epsVR}{\VR^{\epsilon}}

\newcommand{\sminus}{-}

\newcommand{\mhc}{\mathrm{mhc}}

\DeclareMathOperator*{\argmin}{arg\,min}

%% file: abstract.tex
Persistent homology is a popular method for computing topological features of (metric) data. Standard approaches based on the \v{C}ech or Rips filtration are stable under small perturbations of the data, but highly sensitive to outliers. This lack of robustness has been frequently addressed in the literature. In this paper, we take a novel perspective by asking the following question: When can we guarantee that an observed persistent feature (a bar) is inherent to the underlying data in the presence of a limited number of unknown, arbitrary outliers.
We formalize this question by introducing the notion of \emph{adversarial robustness}, and study the problem of deciding whether a given bar in the barcode of a filtered simplicial complex is adversarially robust. We show that this problem is essentially equivalent to a homological variant of the minimum cut problem in simplicial complexes, which we believe to be of independent interest. As our main technical contribution, we provide the first computational complexity results for this problem, consisting of an efficient algorithm in $0$-dimensional homology, NP-hardness for the general problem, and an efficient algorithm for codimension-$1$ in $n$-dimensional complexes embedded in~$\mathbb{R}^n$.
We also analyze its natural linear programming relaxation, whose dual defines a homological analog of the max-flow problem in graphs. 
We show that a max-flow/min-cut theorem does not hold in our setting, implying that the LP relaxation is not tight in general.
Finally, in the special case of the Rips filtration, we provide a global heuristic based on the Hausdorff distance that guarantees adversarial robustness of sufficiently long bars. This connects adversarial robustness to standard stability theorems in persistent homology.

%% file: Introduction.tex
Persistent homology is a central method in topological data analysis. 
It is used to extract topological features from (metric) data across a range of spatial scales. 
At a high level, it works as follows. First, we represent our data by a nested sequence of simplicial complexes, called a filtration. 
An important example is the (Vietoris-)Rips filtration, which arises from pairwise distances between points in a metric space.
The evolution of the homology groups of the complexes in this filtration (i.e., its \emph{persistent} homology) can be captured by a \emph{barcode}, which consists of a (multi)set of intervals (or \emph{bars}) whose endpoints represent the appearance (\emph{birth}) and disappearance (\emph{death}) of a homology class at a particular step in the filtration.
Barcodes serve as a topological signature of the underlying data set. 
We refer to~\cite{persistencebook, persistencesurvey, persistencesurvey2, computingpersistence} for surveys on persistent homology and its many applications. 

\subparagraph*{Stability and robustness.}
An important property of barcodes is their \emph{stability} under small perturbations:
For example, the \textit{Bottleneck distance} between the barcodes of Rips filtrations may be bounded in terms of the \emph{(Gromov-)Hausdorff distance} between their underlying metric spaces~\cite{originalstability}.
On the other hand, barcodes are (in)famously sensitive to \emph{outliers}. Indeed, even a single outlier may cause arbitrarily large changes in the barcode, which makes persistent homology unreliable in the presence of noise.
This issue has been frequently addressed in the literature; we highlight three approaches. 
First, one may consider alternative metric interpretations of the data which take density into account, and whose resulting barcodes are (hopefully) more robust.
For instance, a filtration based on the \emph{distance-to-measure} function~\cite{geominference} achieves stability with respect to the \emph{Wasserstein distance}, which tolerates small amounts of outliers~\cite{Buchet2016}. 
See~\cite{CDforTDA, kerneldistance} for similar approaches using kernel functions. A downside is that these filtrations are difficult to interpret geometrically, often depend on some choice of secondary parameters, and are computationally expensive, making them less suited for topological inference than the Rips filtration.
A second approach is to (cleverly) subsample the data to determine so-called \emph{landmarks}, and then construct a filtration based on these landmarks. An example is the \emph{(lazy) witness filtration}~\cite{Witnesscomplex}. 
Empirically, it appears that the landmarks may be chosen in a way that reduces the sensitivity to outliers (in the original data); see~\cite{Stolz2023}. However, this effect is hard to quantify theoretically. 
Third, one may consider filtrations indexed by multiple parameters, modeling for example both scale and density. The persistent homology of such bifiltrations can be provably robust to certain types of noise~\cite{Blumberg2022}. 
However, multiparameter persistence modules are significantly more difficult to represent than their one-parameter counterparts. In particular, they generally do not have a barcode, which is a serious theoretical and practical drawback; see~\cite{botnan2023introductionmultiparameterpersistence}.

\subparagraph{A new notion of robustness.}
In this paper, we take a different perspective. Rather than modify existing filtrations to increase their tolerance to noise, we ask the following question.

\begin{question} \label{Q:main}
When can we guarantee that an observed persistent feature (a bar) arising from a filtration is inherent to the underlying data in the presence of (adversarial) noise? 
\end{question}
To address this question, we introduce a new notion of outlier-robustness of persistent topological features (bars) arising from a filtration of a simplicial complex, which we call \emph{adversarial robustness}. Before giving a formal definition, it is illustrative to first consider the Rips filtration on a (finite) set of metric data.
Intuitively speaking, we say that a persistent feature of the Rips filtration is $k$-adversarially robust if it ``continues to exist'' after removing any $k$ points from the data set. In this way, $k$-adversarial robustness certifies that a feature is \emph{inherent} to an underlying data set even after adding (at most~$k$) unknown, arbitrary outliers. 
Alternatively, the largest $k$ for which a feature is $k$-adversarially robust can be thought of as a \emph{measure of robustness} of that feature (possibly after dividing by the total number of data points). With respect to the earlier work outlined above, a key advantage of our approach is that we make no structural assumptions on the outliers.

In what follows, we give a formal definition of adversarial robustness. The key ingredient is the induced matching between barcodes~\cite{Bauer2013InducedMA}, which allows us to relate the persistent features of a filtered simplicial complex to those of its filtered subcomplexes. 
Then, we study the computational problem of deciding whether a given bar is adversarially robust. To this end,  we first show that adversarial robustness of a bar can be determined by solving a homological variant of the minimum cut problem in a single (unfiltered) simplicial complex. Thus, it suffices to study the complexity of that problem, which we believe to be of independent interest. We give efficient algorithms for zero-dimensional homology and complexes embedded in $\R^n$. On the other hand, we show that the general problem is NP-hard. We also analyze its natural linear programming
relaxation, whose dual defines a homological analog of the max-flow problem in graphs. We conclude by giving an efficiently computable heuristic for adversarial robustness in Rips filtrations, which is connected to Hausdorff stability.

\subsection{Adversarially robust persistent features}
Let $\simp$ be a (finite) simplicial complex, and let $\fsimp = (K_i)_{0 \leq i \leq m}$ be a filtration of~$K$, i.e., a sequence of simplicial complexes $\emptyset = K_0 \subseteq K_1 \subseteq \ldots \subseteq K_m = K$. 
Throughout, we assume that $\fsimp$ is a \emph{simplex-wise} filtration, i.e., for any $i < m$, the complex $K_{i+1}$ is obtained by adding (at most) a single simplex to $K_i$. Any filtration can be made simplex-wise by breaking ties (arbitrarily, but so that face relations are preserved) whenever multiple simplices are added at once, which is also what is done in practice.
For $p \geq 0$, we denote the $p$-dimensional persistent homology of $\fsimp$ by $\PH_p(\fsimp)$, and we write $\bc(\PH_p(\fsimp))$ for its barcode.
For $A \subseteq \simp$, we write $\simp - A$ for the largest subcomplex of $\simp$ contained in $\simp \setminus A$. That is, $\simp - A$ is  obtained from~$\simp$ by removing all simplices that have a face in $A$. Similarly, we write $\fsimp - A$ for the filtration of $\simp - A$ given by $(\simp_i - A)_{0 \leq i \leq m}$.  
There is a natural way to relate the barcodes associated with $\fsimp$ and~${\fsimp - A}$. Namely, the inclusion $\simp-A \hookrightarrow \simp$ induces a map $\PH_p(\fsimp -A) \to \PH_p(\fsimp )$. In turn, this map induces a (partial) matching between the respective barcodes, which we denote $\im_{\fsimp-A \hookrightarrow \fsimp}$. 
This is a special case of the so-called \emph{induced matching}, which plays a crucial role in a proof of the algebraic stability theorem~\cite{Bauer2013InducedMA}. 
Generally, the induced matching is \emph{not} functorial, but in the context of simplex-wise filtrations it allows us to unambiguously relate bars in the barcodes of $\PH_p(\fsimp)$ and $\PH_p(\fsimp - A)$; see~\Cref{SEC:inducedmatching} for more details. 
This allows us to give the central definition of this paper. For $s \in \N$, let $\simp^{(s)} \subseteq \simp$ denote the set of $s$-simplices in $\simp$.
\begin{definition}[Adversarial robustness] \label{DEF:Robustness}
    Let $\fsimp$ be a simplex-wise filtration of a simplicial complex $K$. Let $p \geq 0$, and let $s \leq p$. A bar $B \in \bc(\PH_p(\fsimp))$ is \emph{$k$-adversarially robust (in degree~$s$)} if, for each subset $A \subseteq K^{(s)}$ of size at most $k$, 
    we have 
    $
        B \in  \mathrm{Im}(\im_{\fsimp -A \hookrightarrow \fsimp}).
    $
\end{definition}
\Cref{DEF:Robustness} guarantees that a bar is ``present'' in any (filtered) subcomplex of $\simp$ missing at most $k$ $s$-simplices. If $\fsimp$ is a Rips filtration, the $0$-simplices (vertices) of $K$ correspond to metric data points. Thus, $k$-adversarial robustness (in degree $0$) of a bar means that it is present in the Rips filtration of any subset of the original data obtained by removing at most~$k$ points, matching our earlier intuitive description.

\subsection{(Minimum) homological cuts}
Adversarial robustness of bars is closely related to a novel homological variant of the min-cut problem in simplicial complexes, which we introduce in this work.  For $C \subseteq K^{(s)}$, we write $\ihomo_{K \sminus C \hookrightarrow K} : H_p(K \sminus C) \rightarrow H_p(K)$ for the map induced by the inclusion $K \sminus C \hookrightarrow K$.
\begin{definition}[homological cuts] \label{DEF:homcut}
    Let $K$ be a simplicial complex. Let $p \geq 0$, and $s \leq p$. We say that $C \subseteq \simp^{(s)}$ is a \emph{homological $s$-cut} for~$\gamma \in H_p(K)$ if
    $
        \gamma \not \in \Ima(\ihomo_{K \sminus C \hookrightarrow K}).
    $ We refer to the special cases $s=0, s=1$ as homological \emph{vertex} and \emph{edge} cuts, respectively; see~\Cref{FIG:CUTEXAMPLE}.
\end{definition} 

\begin{figure}[h]
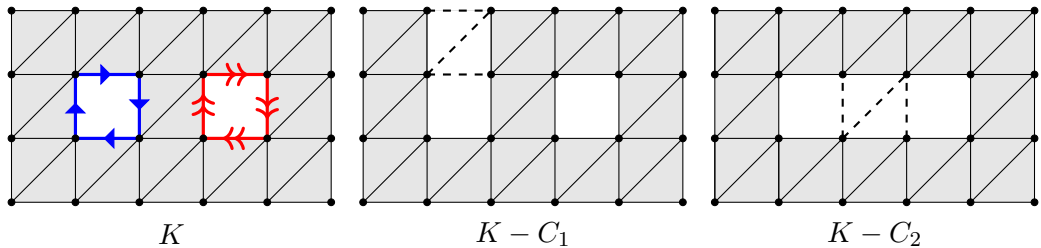

    \centering
    \include{tikz/cutexample}
    \vspace{-1cm}
    \caption{A simplicial complex $K$ with cycles $c_{\rm left}, c_{\rm right} \in C_1(K; \R)$ drawn in blue single arrows and red double arrows, respectively (all coefficients equal to~$1$). The classes $[c_{\rm left}]$ and $[c_{\rm right}]$ generate $H_1(K; \R) \cong \R^2$. 
    On the right: two subcomplexes obtained by removing subsets $C_1, C_2 \subseteq K^{(1)}$ (dashed) from $K$, respectively. Note that $C_1$ is an edge cut for $[c_{\rm left}]$ and $[c_{\rm left} + c_{\rm right}]$, but not for $[c_{\rm right}]$. On the other hand, $C_2$ is a $1$-cut for both $[c_{\rm left}]$ and $[c_{\rm right}]$, but not for $[c_{\rm left} + c_{\rm right}]$.
    }
    \label{FIG:CUTEXAMPLE}
\end{figure}

To connect homological cuts to adversarial robustness, we use the fact that, for a simplex-wise filtration $\fsimp$, each bar $B \in \bc(\PH_p(\fsimp))$ corresponds to a pair $(\sigma_B, \tau_B)$ of simplices, whose insertions at steps $i=b$ and $i = d$, respectively, represent the birth and death of any cycle representing $B$. 
These are called \emph{persistence pairs}. We call the complex $\predeath{B} := K_{d-1}$ the \emph{predeath complex} of $B$. Now, the following proposition shows that adversarial robustness of~$B$ is characterized by homological cuts of $[\tau_B]$ in $K_B$. We give its proof in~\Cref{SEC:proofequiv}.
\begin{proposition} \label{THM:main:LOC}
    Let $p \geq 0$, and let $B \in \bc(\PH_p(\fsimp))$ be a bar in the barcode of a simplex-wise filtration $\fsimp$ of a simplicial complex $K$. Let $\predeath{B}$, $\deathsimplex{B}$ be its pre-death complex and death simplex. Then, $B$ is $k$-adversarially robust (in degree $s$) if, and only if,
    \[
        [\partial \deathsimplex{B}] \in \mathrm{Im}(\ihomo_{\predeath{B} \sminus A \hookrightarrow \predeath{B}}) \text{ for all $A \subseteq \simp^{(s)}$ with $|A| \leq k $}.
    \]
    Thus, $B$ is $k$-adversarially robust iff all homological $s$-cuts of $[\deathsimplex{B}]$ in $\predeath{B}$ have size at least~$k$.
\end{proposition}
The above proposition motivates our study of the following problem.
\begin{problem}
The \emph{minimum homological $s$-cut problem} asks to compute
    \begin{equation}
        \mhc(K, \gamma, s) := \min_{C \subseteq K^{(s)}} \left\{ |C| : \text{$C$ is a homological $s$-cut for $\gamma$} \right\}.     
        \tag{MHC} \label{EQ:MHC}
    \end{equation}
    We call a set $C$ attaining the minimum above a \emph{minimum homological $s$-cut} for $\gamma$ (in $K$).
\end{problem}
Apart from its connection to adversarial robustness, we believe that the minimum homological cut problem is of intrinsic interest. It is related to, \emph{but distinct from}, two types of well-studied problems in computational topology. On the one hand, there are problems related to finding a smallest (or otherwise `optimal') representing cycle of a homology class in a simplicial complex~\cite{Blaser:paramboundedchains, BWA:MinBoundedChains, Chambers:MinArea}; this is often referred to as \emph{homology localization}. By contrast, the problem we consider can be thought of as \emph{co}homology localization.
The computational complexity of homology localization varies based on the choice of coefficient field, the dimension of the homology group, and additional assumptions on the underlying complex.
For example, while the general problem is hard~\cite{ChenFreedman2011}, a linear programming relaxation yields an efficient algorithm for finding minimum weight homologous cycles over integer coefficients for complexes whose boundary matrix is totally unimodular~\cite{optimalchainLP}.
On the other hand, there are ordinary min-cut (and max-flow) problems in graphs with topological structure,
e.g., graphs that can be embedded in a surface of low genus. 
There, (co)homological properties of cuts (or flows) can be used to achieve algorithmic speedups with respect to the general case~\cite{Erickson:surfacegraphs, Chambers2012, Erickson:homcovers}.

\subparagraph*{Linear programming relaxations.}
\label{SEC:INTRO:LP} As we explain in~\Cref{SEC:LP}, homological edge cuts of a class $\gamma \in H_1(K; \R)$ are naturally related to (ordinary) cuts in graphs. Namely, we show that a minimum homological edge cut can be found by optimizing the number~$\|\varphi\|_0$ of non-zero coefficients of a vector $\varphi$ (indexed by the edges of $K$) under a set of linear constraints involving the boundary matrix of $K$ (\Cref{PROP:ell0}). 
The relaxation of this problem obtained by optimizing the $1$-norm $\|\varphi\|_1$ instead is a linear program. 
Its dual may be interpreted as a homological analog of the max-flow problem in graphs. Contrary to the graph setting, we show that there is no max-flow min-cut theorem in our case: homological max-flows are not necessarily integral, meaning the LP relaxation is not tight (\Cref{EXMP:mincutmaxflow}). 

\subsection{Main contributions}

\subparagraph{Complexity of the minimum homological cut problem.} 
We prove positive and negative results on the computational complexity of the minimum homological cut problem~\eqref{EQ:MHC} introduced above, in terms of the number of simplices $|\simp|$ in the complex $K$. By~\Cref{THM:main:LOC}, these results have immediate implications on the complexity of determining adversarial robustness of bars.
Our first contribution is an efficient algorithm for the case where $\gamma \in H_0(K)$ is a $0$-dimensional homology class, i.e., $p=0$.
\begin{restatable}{theorem} {Hzeroeasy}\label{THM:intro:H0easy}
Let $K$ be a simplicial complex , and let $\gamma \in H_0(K)$. We can compute a minimum homological vertex cut for $\gamma$ in time $O(|K^{(0)}| + |K^{(1)}|)$.
\end{restatable}
\Cref{THM:intro:H0easy} follows from the fact that a minimum homological cut of a class in $H_0$ is always equal to the vertex set of a connected component of $K$, as we show in \Cref{SEC:H0easy}.

Next, we show that finding a minimum homological cut is NP-hard, already when $p = 1$. 
\begin{theorem} \label{THM:NPhardness}
    For $p = 1$ and $s=1$, the minimum homological $s$-cut problem is NP-hard. This is true in particular when homology is taken with coefficients in  $\R$ or $\Z/2\Z$.
\end{theorem}
We give the proof of \Cref{THM:NPhardness} in \Cref{SEC:NPhardness}. It relies on a reduction from \textsc{Exact Cover by 3-Sets (X3C)}.
In~\cite[pp. 246]{GareyJohnson}, \textsc{X3C} is used to show hardness of the problem of finding minimum weight solutions to linear equations. It was also used in~\cite{hardnessofL0} to show hardness of finding the \emph{sparsest} approximate solution to a set of linear equations.
As mentioned, we show in~\Cref{SEC:LP} that~\eqref{EQ:MHC} can be solved by finding the sparsest solution to a particular set of linear equations. In light of these observations, \textsc{X3C} is a natural candidate for showing hardness of~\eqref{EQ:MHC}. 
In our proof, we construct for any X3C-instance an equivalent instance of~\eqref{EQ:MHC}. Our construction relies on ``gluing together'' punctured discs, each representing a $3$-set of the X3C-instance, to obtain a topological space which is then triangulated carefully to ensure minimum homological cuts correspond to exact covers. This approach resembles earlier work on NP-hardness in homology localization~\cite{Agoletal2005, ChenFreedman2011, DunfieldHirani2011}. A distinction is that these mostly rely on reductions from SAT-problems; to the best of our knowledge, a reduction from X3C was not considered before in this context.
We remark that our construction can be extended to cover the case $s=0$; see \Cref{REM:Hardness_vertex_cuts}. 
Via~\Cref{THM:main:LOC}, our result also implies that testing $k$-adversarial robustness of bars (in degree $0$ and $1$) is hard in general.

Finally, we give an efficient algorithm for the following special case, which covers, e.g., subcomplexes of the Delauney triangulation of a finite metric data set in $\R^n$.
\begin{restatable}{theorem}{embeddedeasy} \label{THM:embeddedeasy}
    Suppose that $K$ is an $n$-dimensional simplicial complex embedded in $\R^n$, and let $\gamma \in H_{n-1}(K)$. We can compute a minimum homological $(n-1)$-cut for $\gamma$ in polynomial time in the number of simplices of $K$.
\end{restatable}
Our proof of \Cref{THM:embeddedeasy} in \Cref{SEC:efficientalgorithm} relies on Alexander duality, which can be used to construct the so-called \emph{extended dual graph} of~$K$. 
In~\cite{dual_graph_dey}, this graph was used to find optimal generators for certain homology classes: these correspond to minimum cuts in the graph. For us, the situation is precisely the opposite: as we will show, minimum homological cuts in $K$ correspond to shortest paths in the extended dual graph (allowing for efficient computation). The use of the extended dual graph in homology localization dates back to~\cite{Sullivan:thesis}; see also~\cite{Buehler2002}.

\subparagraph*{The Rips filtration.}
Complementing our hardness result, we give an efficiently computable \emph{heuristic} based on the Hausdorff distance to test adversarial robustness of bars arising from the Rips filtration of a metric data set $X$. Namely, we show that any bar in the barcode of the Rips filtration of length at least
    \[
        \hhy := \max_{A \subseteq X, ~|A| \leq k} d_{H}(X \setminus A, X)
    \]
    is automatically $k$-adversarily robust (in degree $0$)\footnote{The Rips filtration is not simplex-wise, therefore, as mentioned, we must break ties to obtain a simplex wise filtration to apply \Cref{DEF:Robustness}. This result does not depend on the choice of tiebreaker.}; see~\Cref{SEC:APPLICATION:HH} for a precise statement. Moreover, we show that the parameter~$\hhy$ can be computed in time $O(|X|^2 \log |X|)$. This result links adversarial robustness to classical stability results in persistent homology. 
    It shows that our definition is at least as expressive as a naive definition based only on the length of bars and the Hausdorff stability of the Rips filtration. That is, sufficiently long bars are outlier-robust according to our definition (as one would expect). But importantly, our definition is also capable of recognizing relatively short bars as robust; see \Cref{FIG:krobustexample}.

\begin{figure}[h]
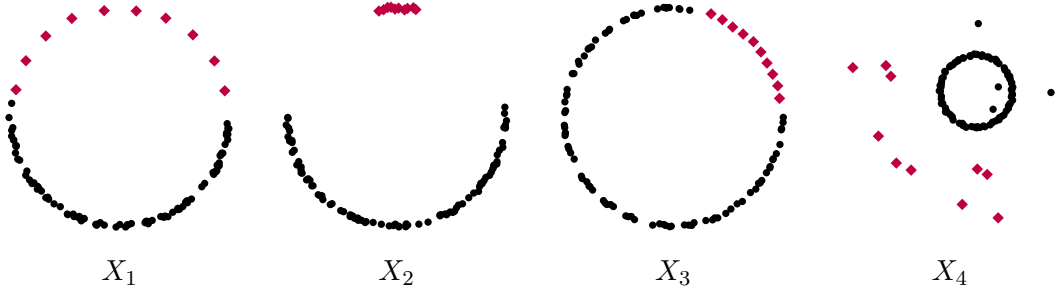

\centering
\include{tikz/krobustexample}
\caption{Four data sets in $\R^2$, each of size $100$, whose Rips filtrations each induce a $1$-dimensional persistent feature. For $X_1$ and $X_2$, these features are not $k$-adversarially robust for $k=10$, evidenced by the subsets $A_1 \subseteq X_1$ and $A_2 \subseteq X_2$ marked in red. The points in $A_2$ are quite dense, and so standard subsampling techniques will likely not remove them. On the other hand, the feature is $10$-robust in $X_3$, as its length exceeds $\mathcal{H}_{X_3, 10}$, which equals $d_H(X_3 \setminus A_3, X_3)$ for the subset $A_3 \subseteq X_3$ marked in red (cf.~\Cref{THM:main:HH}). In $X_4$, the  feature induced by the densely sampled circle is $10$-robust, even though its bar has length strictly less than $\mathcal{H}_{X_4, 10}$ (evidenced by the set $A_4$ marked in red). 
}
\label{FIG:krobustexample}
\end{figure}

%% file: tikz/cutexample.tex
\begin{tikzpicture}[scale=0.85]

\tikzset{
  doublemidarrow/.style={
    decoration={
      markings,
      mark=at position 0.7 with {\arrow{>>}}
    },
    postaction={decorate}
  }
}

\begin{scope}
\foreach \x in {0,1,2,3,4,5} {
    \foreach \y in {0,1,2,3} {
        \coordinate (v\x\y) at (\x,\y);
    }
}

\draw (2.5, -0.5) node {$K$};

\fill[gray!20] (v00) -- (v50) -- (v53) -- (v03) -- (v00);
\fill[white] (v11) -- (v12) -- (v22) -- (v21) -- (v11);
\fill[white] (v31) -- (v32) -- (v42) -- (v41) -- (v31);

\foreach \y in {0,1,2,3} {
    \draw (v0\y) -- (v5\y);
}
\foreach \x in {0,1,2,3,4,5} {
    \draw (v\x0) -- (v\x3);
}

\draw (v00) -- (v11);
\draw (v01) -- (v12);
\draw (v02) -- (v13);

\draw (v10) -- (v21);
\draw (v12) -- (v23);

\draw (v20) -- (v31);
\draw (v21) -- (v32);
\draw (v22) -- (v33);

\draw (v30) -- (v41);
\draw (v32) -- (v43);

\draw (v40) -- (v51);
\draw (v41) -- (v52);
\draw (v42) -- (v53);

\draw[blue, sloped, allow upside down, very thick] (v11) -- node {\midarrow} (v12);
\draw[blue, sloped, allow upside down, very thick] (v12) -- node {\midarrow} (v22);
\draw[blue, sloped, allow upside down, very thick] (v22) -- node {\midarrow} (v21);
\draw[blue, sloped, allow upside down, very thick] (v21) -- node {\midarrow} (v11);

\draw[red, sloped, allow upside down, very thick, doublemidarrow] (v31) -- node {} (v32);
\draw[red, sloped, allow upside down, very thick, doublemidarrow] (v32) -- node {} (v42);
\draw[red, sloped, allow upside down, very thick, doublemidarrow] (v42) -- node {} (v41);
\draw[red, sloped, allow upside down, very thick, doublemidarrow] (v41) -- node {} (v31);

\foreach \x in {0,1,2,3,4,5} {
    \foreach \y in {0,1,2,3} {
        \filldraw (v\x\y) circle (1.5pt);
    }
}


\end{scope}

\begin{scope}[xshift=5.5cm, yshift=0cm]
\foreach \x in {0,1,2,3,4,5} {
    \foreach \y in {0,1,2,3} {
        \coordinate (v\x\y) at (\x,\y);
    }
}

\draw (2.5, -0.5) node {$K - C_1$};

\fill[gray!20] (v00) -- (v50) -- (v53) -- (v03) -- (v00);
\fill[white] (v11) -- (v12) -- (v22) -- (v21) -- (v11);
\fill[white] (v31) -- (v32) -- (v42) -- (v41) -- (v31);


\foreach \y in {0,1} {
    \draw (v0\y) -- (v5\y);
}
\foreach \y in {2,3} {
    \draw (v0\y) -- (v1\y);
    \draw (v2\y) -- (v5\y);
}

\foreach \x in {2,3,4,5} {
    \draw (v\x0) -- (v\x3);
}

\draw (v01) -- (v02);
\draw (v11) -- (v12);
\draw (v01) -- (v12);

\draw (v02) -- (v03);
\draw (v00) -- (v01);

\draw (v12) -- (v13);
\draw (v10) -- (v11);

\draw (v00) -- (v11);
\draw (v02) -- (v13);

\draw (v10) -- (v21);
\draw (v12) -- (v23);

\draw (v20) -- (v31);
\draw (v21) -- (v32);
\draw (v22) -- (v33);

\draw (v30) -- (v41);
\draw (v32) -- (v43);

\draw (v40) -- (v51);
\draw (v41) -- (v52);
\draw (v42) -- (v53);

\fill[white] (v12) -- (v13) -- (v23) -- (v22) -- cycle;
\draw[dashed, thick] (v12) -- (v23);
\draw[dashed, thick] (v12) -- (v22);
\draw[dashed, thick] (v13) -- (v23);
\draw (v12) -- (v13);
\draw (v22) -- (v23);

\foreach \x in {0,1,2,3,4,5} {
    \foreach \y in {0,1,2,3} {
        \filldraw (v\x\y) circle (1.5pt);
    }
}


\end{scope}

\begin{scope}[xshift=11cm, yshift=0cm]
\foreach \x in {0,1,2,3,4,5} {
    \foreach \y in {0,1,2,3} {
        \coordinate (v\x\y) at (\x,\y);
    }
}

\draw (2.5, -0.5) node {$K - C_2$};

\fill[gray!20] (v00) -- (v50) -- (v53) -- (v03) -- (v00);
\fill[white] (v11) -- (v12) -- (v22) -- (v21) -- (v11);
\fill[white] (v31) -- (v32) -- (v42) -- (v41) -- (v31);

\fill[white] (v21) -- (v31) -- (v32) -- (v22) -- cycle;

\foreach \y in {0,1,2,3} {
    \draw (v0\y) -- (v5\y);
}
\foreach \x in {0,1,4,5} {
    \draw (v\x0) -- (v\x3);
}

\draw (v22) -- (v23);
\draw (v32) -- (v33);

\draw (v20) -- (v21);
\draw (v30) -- (v31);

\draw[dashed, thick] (v21) -- (v22);
\draw[dashed, thick] (v21) -- (v32);
\draw[dashed, thick] (v31) -- (v32);

\draw (v02) -- (v03);
\draw (v00) -- (v01);

\draw (v12) -- (v13);
\draw (v10) -- (v11);

\draw (v00) -- (v11);
\draw (v01) -- (v12);
\draw (v02) -- (v13);

\draw (v10) -- (v21);
\draw (v12) -- (v23);

\draw (v20) -- (v31);
\draw (v22) -- (v33);

\draw (v30) -- (v41);
\draw (v32) -- (v43);

\draw (v40) -- (v51);
\draw (v41) -- (v52);
\draw (v42) -- (v53);


\foreach \x in {0,1,2,3,4,5} {
    \foreach \y in {0,1,2,3} {
        \filldraw (v\x\y) circle (1.5pt);
    }
}


\end{scope}

\end{tikzpicture}

%% file: tikz/krobustexample.tex
\begin{tikzpicture}[scale=0.92]
    \begin{scope}[scale=0.75]
    \pgfmathsetseed{1}
    \def\radius{3}
    \def\n{90}
    \def\m{10}

    \def\noise{0.10}

    \foreach \i in {1,...,\n} {
        \pgfmathsetmacro{\angle}{180/(\n-1)*(\i-1)}
        \pgfmathsetmacro{\angleWithNoise}{\angle + (rnd-0.5)*\noise*180}  \pgfmathsetmacro{\radiusWithNoise}{\radius + (rnd-0.5)*\noise}       
        \coordinate (P\i) at (-\angleWithNoise:\radiusWithNoise*0.7);
        \fill(P\i) circle (2pt);
    };
    \foreach \j in {1,...,\m} {
        \pgfmathsetmacro{\anglem}{150/(\m-1)*(\j-1)}
        \pgfmathsetmacro{\angleWithNoisem}{\anglem + (rnd-0.5)*\noise*20 + 15}  \pgfmathsetmacro{\radiusWithNoisem}{\radius + (rnd-0.5)*\noise}       
        \coordinate (Pm\j) at (\angleWithNoisem:\radiusWithNoisem*0.7);
        \node[diamond,fill, scale=0.3, purple] (d) at (Pm\j) {};
    };
    \node at (0,-3) {$X_1$};
    \end{scope}

    \begin{scope}[xshift=4cm, scale=0.75]
    \pgfmathsetseed{2}

    \def\radius{3}
    \def\n{90}
    \def\m{10}

    \def\noise{0.10}

    \foreach \i in {1,...,\n} {
        \pgfmathsetmacro{\angle}{180/(\n-1)*(\i-1)}
        \pgfmathsetmacro{\angleWithNoise}{\angle + (rnd-0.5)*\noise*180}  \pgfmathsetmacro{\radiusWithNoise}{\radius + (rnd-0.5)*\noise}       
        \coordinate (P\i) at (-\angleWithNoise:\radiusWithNoise*0.7);
        \fill (P\i) circle (2pt);
    };
    \foreach \j in {1,...,\m} {
        \pgfmathsetmacro{\anglem}{20/(\m-1)*(\j-1)}
        \pgfmathsetmacro{\angleWithNoisem}{\anglem + (rnd-0.5)*\noise*20 + 80}  \pgfmathsetmacro{\radiusWithNoisem}{\radius + (rnd-0.5)*\noise}       
        \coordinate (Pm\j) at (\angleWithNoisem:\radiusWithNoisem*0.7);
        \node[diamond,fill, scale=0.3, purple] (d) at (Pm\j) {};
    };
    \node at (0,-3) {$X_2$};
    \end{scope}

    \begin{scope}[xshift=8cm, scale=0.75]

    \pgfmathsetseed{4}

    \def\radius{3}
    \def\n{90}
    \def\m{10}

    \def\noise{0.07}

    \foreach \i in {1,...,\n} {
        \pgfmathsetmacro{\angle}{280/(\n-1)*(\i-1)}
        \pgfmathsetmacro{\angleWithNoise}{\angle + (rnd-0.5)*\noise*140}  \pgfmathsetmacro{\radiusWithNoise}{\radius + (rnd-0.5)*\noise}       
        \coordinate (P\i) at (-\angleWithNoise:\radiusWithNoise*0.7);
        \fill (P\i) circle (2pt);
    };
    \foreach \j in {1,...,\m} {
        \pgfmathsetmacro{\angle}{60/(\m-1)*(\j-1)}
        \pgfmathsetmacro{\angleWithNoise}{\angle + 290}  \pgfmathsetmacro{\radiusWithNoise}{\radius + (rnd-0.5)*\noise}       
        \coordinate (P\j) at (-\angleWithNoise:\radiusWithNoise*0.7);
        \node[diamond,fill, scale=0.3, purple] (d) at (P\j) {};
    };
    \node at (0,-3) {$X_3$};
    \end{scope}

\begin{scope}[xshift=12cm, yshift=0cm, scale=0.75]

    \pgfmathsetseed{8}

    \foreach \i in {1,...,5}{%
      \pgfmathsetmacro{\x}{2.5*rand}%
      \pgfmathsetmacro{\y}{2.5*rand}%
      \node[diamond,fill, scale=0.3, purple] at (-1.5 + 0.25*\x, 0.5*\y) {};
      \node[diamond,fill, scale=0.3, purple] at (0.5*\x, -1.5 + 0.25*\y) {};
  }

      \foreach \i in {1,...,5}{%
      \pgfmathsetmacro{\x}{rand}%
      \pgfmathsetmacro{\y}{rand}%
      \fill (\x + 1, \y + 1) circle (2pt);
  }

    \def\radius{1}
    \def\n{85}

    \def\noise{0.07}

    \begin{scope}[xshift=0.5cm, yshift=0.5cm]
    \foreach \i in {1,...,\n} {
        \pgfmathsetmacro{\angle}{360/(\n-1)*(\i-1)}
        \pgfmathsetmacro{\angleWithNoise}{\angle + (rnd-0.5)*\noise*140}  \pgfmathsetmacro{\radiusWithNoise}{\radius + (rnd-0.5)*\noise}       
        \coordinate (P\i) at (-\angleWithNoise:\radiusWithNoise*0.7);
        \fill (P\i) circle (2pt);
    };
    \end{scope}

    \node at (0,-3) {$X_4$};
    \end{scope}


    






    
\end{tikzpicture}

%% file: Preliminaries.tex
\subsection{Simplicial complexes}
An (abstract) simplicial complex (on a set $Y$) is a collection $\simp$ of non-empty subsets $\sigma \subseteq Y$, called \emph{simplices}, such that if $\tau \in \simp$ and $\sigma \subseteq \tau$, then $\sigma \in \simp$. In that case, $\sigma$ is called a \emph{face} of $\tau$, and $\tau$ is called a \emph{coface} of $\sigma$. 
For $A \subseteq K$, we denote by $K \sminus A$ the largest subcomplex of $K$ which does not contain any simplices in $A$ (which is obtained by removing all simplices from $K$ which have a face in $A$). For $s \in \N$, we write $K^{(s)}$ for the set of $s$-simplices of~$K$, being the simplices $\sigma \in K$ with $|\sigma| = s+1$.
We write $V(K) := K^{(0)}$ for the vertices of $K$, and $E(K) := K^{(1)}$ for its edges. If there is a $p$ such that $K^{(p)} \neq \emptyset $ and $K^{(q)} = \emptyset$ for all $q>p$, then we say that $K$ is $p$-dimensional, and we call $K$ a $p$-complex. 
If $K$ is a $p$-complex and each simplex of $K$ is contained in a $p$-simplex, then $K$ is called a \emph{pure} $p$-complex.

\subparagraph*{Embedded complexes.} 
A \emph{geometric} $p$-simplex $\sigma$ in $\R^n$ is the convex hull of $p+1$ affinely independent vectors $v_0,\ldots,v_p \in \R^n$. The elements of $V(\sigma) := \{v_0,\ldots,v_p\}$ are called the vertices of $\sigma$. 
If $\tau$ is the convex hull of a subset of the vertices of $\sigma$, then $\tau$ is called a \emph{face} of $\sigma$. 
A \emph{geometric} simplicial complex $\simp$ in $\R^n$ is a collection of geometric simplices in $\R^n$ with the properties that 1) the face of every simplex is in $\simp$; 2) for every pair of simplices $\sigma,\tau \in \simp$ with non-empty intersection, $\sigma \cap \tau$ is in $\simp$. We denote the \emph{geometric realization} of $\simp$ by $\|\simp\| := \bigcup_{\sigma \in \simp} \sigma \subseteq \R^n$, i.e. the subset of $\R^n $  consisting of the union of simplices that make up $\simp$. 
If $\simp$ is a geometric simplicial complex, then $A(\simp) := \{V(\sigma) \, | \, \sigma \in \simp\}$ is an abstract simplicial complex. 
We say that an abstract simplicial complex $L$ is \emph{embedded in $\R^n$} if $L \cong A(\simp)$ for some geometric simplicial complex $\simp$ in $\R^n$, and we define the geometric realization of the embedded simplicial complex $L$ to be $\| L \| := \| \simp \|$.

\subparagraph*{Simplicial  homology.}
For a field $\mathbb{F}$, and $p \in \N$, we denote by $C_p(K;\mathbb{F})$ the simplicial $p$-chains on $K$.
We write $\partial_p : C_p(K;\mathbb{F}) \to C_{p-1}(K;\mathbb{F})$ for the boundary operator.
We have subspaces of cycles $Z_p(K;\mathbb{F}) = \mathrm{Ker } (\partial_p)$ and boundaries $B_p(K;\mathbb{F}) = \mathrm{Im } (\partial_{p+1})$, and the simplicial homology of $K$ is defined as $H_p(K;\mathbb{F}) = Z_p(K;\mathbb{F}) / B_p(K;\mathbb{F})$. Throughout, if we omit $\mathbb{F}$ from the notation, we mean homology over any field. For two complexes $K \subseteq K'$, we denote the inclusion $K \hookrightarrow K'$  by~$\incl_{K \hookrightarrow K'}$. Then, we write 
$
\ihomo_{K \hookrightarrow K'} : H_p(K) \to H_p(K'),~[c] \mapsto [\incl_{K \hookrightarrow K'}(c)]
$
for the induced maps in homology (suppressing $p$ in the notation).

\subsection{Persistent homology}
A persistence module is a functor $M:(\R,\leq ) \to \mathbf{vect}_\mathbb{F}$, where $\mathbf{vect}_\mathbb{F}$ denotes the category of finitely generated vector spaces. A morphism between persistence modules $M$ and $N$ is a natural transformation of functors $f: M \to N$. 
There exists a multiset of intervals $\bc(M)$, called the barcode of $M$, that completely captures $M$ up to isomorphism. We write $\mathbb{I}_{I}$ for the interval module over $I$, i.e, the persistence module which is equal to $\mathbb{F}$ on $I$, connected by identity morphisms, and $0$ elsewhere. Then there is an isomorphism $M \cong \bigoplus_{I \in \bc(M)} \mathbb{I}_I$~\cite{botnan2020decomposition}.

A filtration $\fsimp = (\simp_i)_{0 \leq i \leq m}$ of simplicial complexes naturally defines a persistence module  $\PH_p(\fsimp)$, called the persistent homology of $\fsimp$, which is given at $t\in \R$ by $\PH_p(\fsimp)_t = \mathrm{H}_p(K_{\lfloor t \rfloor})$, and whose structure maps are maps induced by inclusions.
If $\fsimp$ is a simplex-wise filtration, there is at most one birth or death of a bar at each step of the filtration. Hence, any bar $B = [i,j) \in \bc(\PH_p(\fsimp))$ then uniquely corresponds to a pair of simplices $(\sigma_B, \tau_B)$, inserted at steps $i$ and $j$ respectively. 
This is called a \emph{persistence pairing}, and we sometimes refer to the bar as $B = [\sigma_B, \tau_B)$. We denote by $\predeath{B} := K_{d-1} = K_d \setminus \{\tau_B\}$ the \emph{predeath complex} of $B$.
For a filtration $\fsimp$ and a non-decreasing function $f: \N \to \R$, we get a persistence module $\PH_p(\fsimp, f)$ given at time $t \in R$ by $\mathrm{H}_p(K_{\sup\{i : f(i)\leq t\}})$. We think of $f$ as a reparametrization of the filtration. Note that if $f$ is the identity, $\PH_p(\fsimp,f) = \PH_p(\fsimp)$. For each bar $B$ in the barcode of $\PH_p(\fsimp,f)$ there is a bar $[i,j)$ in the barcode of $\PH_p(\fsimp)$ such that $B = [f(i),f(j) )$.

%% file: inducedmatchings.tex
\label{SEC:inducedmatching}
\subsection{Induced matchings} 
While the barcode of a persistence module $M$ is unique, the isomorphism ${M \cong \bigoplus_{I \in \bc(M)} \mathbb{I}_I}$ is not. For a morphism $f:M \to N$, this makes it difficult to relate the barcodes $\bc(M)$ and~$\bc(N)$ to one another in a way that reflects the algebraic structure of $f$. In \cite{Bauer2013InducedMA}, a procedure is described that associates a partial matching $\im_f :\bc(M) \leftrightarrow \bc(N)$ with $f$, in a way that only depends on the barcodes of $M$, $N$, and $\mathrm{Im}(f)$, the image persistence module of $f$. Denote by $\langle \, \cdot \, , d\rangle_M \subseteq \bc(M)$ the bars in the barcode of the form $[ b,d)$ (for simplicity, we assume all intervals in the barcodes of $M$ and $N$ are half-open).  If $f$ is injective, it can be shown that $|\langle \, \cdot \, , d\rangle_M| \leq |\langle \, \cdot \, , d\rangle_N|$.
Moreover, if we order the bars in both $\langle \, \cdot \, , d\rangle_M$ and $\langle \, \cdot \, , d\rangle_N$ by their length, there is an order preserving injection $\im_f^d:\langle \, \cdot \, , d\rangle_M \hookrightarrow \langle \, \cdot \, , d\rangle_N$, with the property that if $\im_f^d([b,d)) = [b',d)$, then $b' \leq b$. Note that if $M$ and $N$ come from simplex-wise filtrations, the sets $\langle \cdot,d\rangle_M$ and  $\langle \cdot,d\rangle_N$ consist of at most one element. The induced matching $\im_f$ is then defined as the union of the matchings $\im_f^d$. If $f$ is surjective, there is a similar way to match bars with the same birth in an order preserving way.
For arbitrary $f :M \to N$, there is a factorization $f : M \twoheadrightarrow \mathrm{Im}(f) \hookrightarrow N$ of $f$ into a composition of a surjective and an injective map. The induced matching $\im_f$ is then defined as the composition of the induced matchings of $M \twoheadrightarrow \mathrm{Im}(f)$ and $\mathrm{Im}(f) \hookrightarrow N$. The induced matching is \emph{not} functorial in~$f$.

A persistence module $M$ is called $\epsilon$-trivial if for each $t \in \R$, the internal morphism $\phi^M_{t \to t+\epsilon}: M_t \to M_{t+\epsilon}$ is the zero morphism. The following theorem is the main result of \cite{Bauer2013InducedMA}.

\begin{theorem}[Induced matching theorem~\cite{Bauer2013InducedMA}] \label{THM:inducedmatching}
    Let $f : M \to N$ be a morphism with $\epsilon$-trivial kernel. Then each bar $[b,d) \in \bc(M)$ with $d-b>\epsilon$ is matched to a bar in $\bc(N)$ by $\im_f$. If $\mathrm{coker}(f)$ is $\epsilon$-trivial, $\im_f$ matches each bar $[b,d) \in \bc(N)$ with $d-b>\epsilon$ to a bar in~$\bc(M)$.
\end{theorem}
We will need this result for our proof of \Cref{THM:main:HH} in~\Cref{SEC:APPLICATION:HH}.

\subsection{Adversarial robustness and minimum homological cuts} \label{SEC:proofequiv}
In this section, we show that $k$-adversarial robustness of a bar can be determined by solving a minimum homological vertex cut problem. That is, we prove~\Cref{THM:main:LOC}.%
\begin{proof}[Proof of~\Cref{THM:main:LOC}]
    Let $B = [b, d)$ be a bar in the barcode of $\PH_p(\fsimp)$, thus with birth simplex $\sigma_B$ inserted at step $b$, and death simplex $\tau_B$ inserted at time $d$. Let $K_B = K_{d-1}$ be its predeath complex.
    Let $A \subseteq K^{(s)}$, and write $\imatch_{\fsimp \sminus A \hookrightarrow \fsimp}$ for the matching induced by the inclusion $\fsimp \sminus A \hookrightarrow \fsimp$. We show that $B \in \Ima(\imatch_{\fsimp \sminus A \hookrightarrow \fsimp})$ if and only $[\partial \tau_B] \in \Ima(\ihomo_{K_B - A \hookrightarrow K_B})$.

    Note that any representative cycle for $B$ is homologous to a scalar multiple of $\partial \deathsimplex{B}$ in $C_p(\predeath{B})$, as $[\partial \deathsimplex{B}]$ generates the kernel of $H_p(\predeath{B}) \to H_p(\predeath{B} \cup \{\deathsimplex{B}\})$.
    If $B \in \mathrm{Im}(\imatch_{\fsimp \sminus A \hookrightarrow \fsimp})$, there is thus a bar $B' =[b', d)$ in the barcode of the image module $\mathrm{Im}(\PH_p(\fsimp - A) \to \PH_p(\fsimp))$.
    Then, any representative cycle $\xi \in C_p(K_{b'} \sminus A) \subseteq C_p(K_{b'})$ for $[b',d)$, when considered as a homology class $[\xi] \in H_p(\predeath{B})$ through the inclusions $K_{b'} \sminus A \hookrightarrow K_{b'} \hookrightarrow \predeath{B}$, is in the kernel of $H_p(K_B) \to H_p(\predeath{B} \cup \{\deathsimplex{B}\})$, and is hence homologous to a scalar multiple of $[\partial \deathsimplex{B}]$.
    
    Conversely, suppose that $[\partial \deathsimplex{B}]$ is in the image of $ H_p(\predeath{B} \sminus A) \to H_p(\predeath{B})$, and let $\xi \in C_p(\predeath{B} \sminus A)$ be a cycle such that $[\xi] = [\partial \deathsimplex{B}] \in H_p(\predeath{B})$. Let $\xi_1,\ldots,\xi_k$ denote representative cycles for summands of $\mathrm{Im}(\PH_p(\fsimp \sminus A) \to \PH_p(\fsimp))$ so that the $[\xi_i]$ are linearly independent in $H_p(\predeath{B})$, and $[\xi] = \lambda_1[\xi_1]+ \ldots + \lambda_k[\xi_k] $. Since $[\xi]$ is a non-zero element in the kernel of the map
    \begin{align*}
        \mathrm{Im}\bigg(H_p(\predeath{B} \sminus A) \to H_p(\predeath{B})\bigg) \to \mathrm{Im}\bigg(H_p((\predeath{B} \cup \{\deathsimplex{B}\}) \sminus A) \to H_p(\predeath{B} \cup \{\deathsimplex{B}\})\bigg),
    \end{align*}
    we see that one of the generators $\xi_j$ must be in this kernel, and thus generate a summand isomorphic to $\mathbb{I}_{[b', d)}$ for some $b \leq b' < d$. This shows that $B$ is in the image of $\imatch_{\fsimp \sminus A \hookrightarrow \fsimp}$. 
\end{proof}

%% file: zero_dimensional.tex
\subsection{An efficient algorithm for zero-dimensional cuts}
\label{SEC:H0easy} Let $\gamma \in H_0(K)$. We show that minimum homological vertex cuts for $\gamma$ have a very particular structure, allowing us to compute them efficiently.
Let $\Gamma_1, \ldots, \Gamma_q \subseteq V(K)$ be the vertex sets of the connected components of~$K$, and let ${c_1 \in V(\Gamma_1), \ldots, c_q \in V(\Gamma_q)}$. Then, $H_0(K)$ is generated by the classes $[c_1], \ldots, [c_q]$, meaning we can write $\gamma = \sum_{i=1}^q \lambda_i [c_i]$, with $\lambda_i \in \mathbb{F}$. 

\begin{proposition} \label{PROP:H0component}
Let $\gamma = \sum_{i=1}^q \lambda_i [c_i] \in H_0(K)$ as in the above. A subset $C \subseteq V(K)$ is a homological vertex cut for $\gamma$ iff $C \supseteq V(\Gamma_\ell)$ for some $1 \leq \ell \leq q$ with $\lambda_\ell \neq 0$. In particular, a homological mincut for $\gamma$ is of the form $C = V(\Gamma_\ell)$, where $\ell = \argmin_i \{ |V(\Gamma_i)| : \lambda_i \neq 0\}$.

\end{proposition}
\begin{proof}
    Write $V = V(K)$. For any $A \subseteq V$, and any $v \in V \setminus A$ we have $\ihomo_{K \sminus A \hookrightarrow K}([v]) = [v]$. Furthermore, for any $v, w \in V$, we have $[v] = [w]$ (in $H_0(K)$) if and only if $v$ and $w$ lie in the same connected component of $K$.    
    Since the classes $[v]$, $v \in V \setminus A$, generate $H_0(K \sminus A)$, we thus have $\gamma \in \mathrm{Im}(\ihomo_{K \sminus A \hookrightarrow K})$ if and only if $V \setminus A$ contains a vertex in each connected component~$\Gamma_i$ of $K$ with $\lambda_i \neq 0$. 
    Thus, $\gamma \not\in \mathrm{Im}(\ihomo_{K \sminus A \hookrightarrow K})$ (i.e., $A$ is a homological cut) iff~$A$ contains all the vertices of at least one connected component $\Gamma_\ell$ of $K$ with $\lambda_\ell \neq 0$.
\end{proof}

\begin{proof}[Proof of~\Cref{THM:intro:H0easy}]
To determine a minimum homological vertex cut of a $0$-dimensional class in a simplicial complex $K$, it suffices by~\Cref{PROP:H0component} to find its connected components. We can do so in time linear in $|V(K)| + |E(K)|$ using a depth-first search.
\end{proof}

%% file: np-hardness.tex
\subsection{NP-hardness of the general problem} \label{SEC:NPhardness} In this section, we prove~\Cref{THM:NPhardness}. That is, we show that the minimum homological cut problem is hard when $p=s=1$. Our arguments apply for homology taken over any field, and even over $\mathbb{Z}$.  We use a reduction from the following NP-hard problem. 
\begin{problem}[Exact cover by 3-sets (X3C)] Given a set $U$, and a collection $\mathcal{S} \subseteq \mathcal{P}(U)$ of subsets of~$U$, each of size exactly $3$, determine whether $\mathcal{S}$ contains an \emph{exact cover} of $U$, i.e., whether there is an $\mathcal{S}' \subseteq \mathcal{S}$ so that each element of $U$ occurs in \emph{exactly one} element of $\mathcal{S}'$.
\end{problem}
We proceed as follows:
To a given instance $(U,\mathcal{S})$ of the X3C problem, we associate a topological space ${X = X(U,\mathcal{S})}$, together with a specific homology class $\gamma \in H_1(X)$ (\Cref{CONSTR:reduction}). 
Then, we construct a  triangulation $K = K(U,\mathcal{S})$ of this space, which satisfies useful technical properties (\Cref{LEM:Triangulation}). 
Finally, we show that if there exists a minimum cut for $\gamma$ in $K$ of a particular size, there exists a solution to the X3C instance, and vice versa. 

\begin{construction}\label{CONSTR:reduction}
    Let $(U,\mathcal{S})$ be an instance of exact cover by 3-sets. 
    For each $u \in U$, denote by $\mathcal{S}_u = \{ S_{u,1} , \ldots S_{u,m_u}\} \subseteq \mathcal{S}$ the elements of $\mathcal{S}$ containing $u$. Let $F_u$ denote the surface obtained from a 2-dimensional disk by removing $|\mathcal{S}_u|$ open disks from its interior. Label the outer boundary of $F_u$ by $u$, and the interior boundary components by the respective elements of $\mathcal{S}$ they correspond to. Choose an orientation for $F_u$, inducing an orientation on its boundary components. Choose generators for the homology groups of the boundary components such that $\partial [F_u] = [u] + [S_{u,1}] + \ldots + [S_{u,m_u}]$; see~\Cref{fig:Surfaces}. 
    
    We obtain $X=X(U,\mathcal{S})$ from $\bigsqcup_{u\in U} F_u $ by gluing together the surfaces in the following way: We identify all outer boundary components with each other along orientation preserving homeomorphisms, so that $[u] = [v]$ in $H_1(X)$ for each $u,v \in U$. We also identify all interior boundary components with each other that have the same label, again in an orientation preserving way. This means that if $S = \{u,v,w\}$, then after gluing $S = S_{u,i_u} = S_{v,i_v} = S_{w,i_w}$, and $[S_{u,i_u}] = [S_{v,i_v}]=[S_{w,i_w}]$ in $H_1(X)$. Finally, we set $\gamma = [u] \in H_1(X)$ for some $u \in U$.
    \end{construction}

\begin{figure}[h]
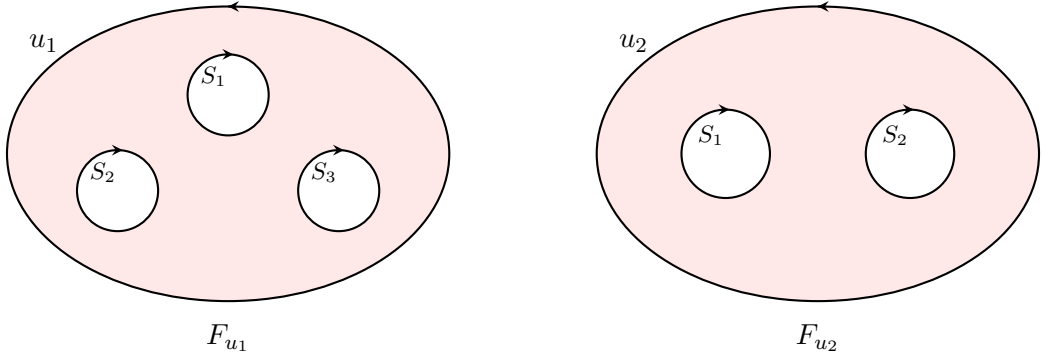

    \centering
    \include{tikz/Surfaces}
    \caption{Surfaces $F_{u_1}$ and $F_{u_2}$ involved in the construction of $X(U,\mathcal{S)}$ for $(U,\mathcal{S}) = (\{u_1,u_2,u_3,u_4\}, \{S_1,S_2,S_3\})$ with $S_1 = \{u_1,u_2,u_3\}, S_2= \{u_1,u_2,u_4\}, S_3 = \{u_1,u_3,u_4\}$. The surfaces are glued along boundary components with the same label, in an orientation preserving way.}
    \label{fig:Surfaces}
\end{figure}

\begin{lemma}\label{LEM:Triangulation}
Let $(U,\mathcal{S})$ be an instance of X3C, and let $k=\max_{u \in U} |\{S \in \mathcal{S}\, | \, u \in S\}| $. Set $c=c(U,\mathcal{S}):= 5\cdot \lceil \log_3(k)\rceil +3$. There exists a triangulation $K = K(U,\mathcal{S})$ for the space $X(U,\mathcal{S})$ of Construction~\ref{CONSTR:reduction}, obtained by constructing a triangulation $L_u$ for each surface $F_u$ and gluing them together, that has the following properties:
    \begin{enumerate}
        \item $K$ is a pure 2-complex consisting of at most $O(|U| \cdot |S|)$ $2$-simplices;
        \item A minimum homology edge cut for $[u] \in H_1(L_u)$ consists of precisely $c+2$ edges; 
        \item A minimum homology edge cut for $[u]$ contains precisely one edge from $u$, and precisely one edge from any of the interior boundary components $S_{u,i}$. For any edge in $u$ and any edge in any $S_{u,i}$, there is a minimum cut for $[u]$ containing both those edges. 
    \end{enumerate}
\end{lemma}
\begin{proof}
    We give an explicit construction of this triangulation in~\Cref{SEC:omittedproofs:hardness}.
\end{proof}
Using~\Cref{LEM:Triangulation}, we can sketch our reduction; a full proof is given in~\Cref{SEC:omittedproofs:hardness}.
\begin{proof}[Proof sketch of~\Cref{THM:NPhardness}]
    We show that there exists a solution to the X3C-instance $(U,\mathcal{S})$ if and only if there exists a cut for $\gamma \in H_1(K(U,\mathcal{S)})$ consisting of at most $c\cdot |U| + \frac{|U|}{3}+1$ edges. If $\mathcal{S^* \subseteq \mathcal{S}}$ is a solution to $(U,\mathcal{S)}$, for each $u\in U$, we pick a cut $C_u$ for $[u] \in H_1(L_u)$ with the property that it contains one edge in the boundary labelled by $u$, and one in the boundary labeled by the set $S_{u} \in S^*$ for which $u \in S_u$. If we consistently choose the same edges in the intersections between surfaces, we can verify that $C = \bigcup_{u\in U} C_u$ is a cut for $\gamma$ consisting of precisely $c\cdot |U| + \frac{|U|}{3}+1$ edges.
    Conversely, suppose $C$ is any cut for $\gamma$. Then it can be shown that for each $u\in U$, $C\cap L_u$ is a cut for $[u] \in H_1(L_u)$. By point 3 of \Cref{LEM:Triangulation}, we then obtain a cover $\mathcal{C}_C$ for $U$. It can be checked that $|C| \geq c \cdot |U| + |\mathcal{C}_C| +1$. It follows that if $|C| \leq c\cdot |U| + \frac{|U|}{3}+1$, then $|\mathcal{C}_C| = \frac{|U|}{3}$, and so $\mathcal{C}_C$ is an exact cover.
\end{proof}

%% file: tikz/Surfaces.tex
\begin{tikzpicture}[scale=0.98]
    \filldraw[fill=red!9, draw=black, thick] (0,0) ellipse (3cm and 2cm); 
    \draw[->, >=stealth, thick] (0,2) arc (89.9:90.1:.01cm); 
    \filldraw[fill=white, draw=black, thick] (-1.5,-.5) circle (.55cm); 
    \draw[->, >=stealth, thick] (-1.48,0.05) -- (-1.43,0.05);
    \filldraw[fill=white, draw=black, thick] (1.5,-.5) circle (.55cm);
    \draw[->, >=stealth, thick] (1.52,0.05) -- (1.57,0.05);
    \filldraw[fill=white, draw=black, thick] (0,.8) circle (.55cm);
    \draw[->, >=stealth, thick] (0.02,1.35) -- (0.07,1.35);
    \node[] at (-2.5,1.5) {$u_1$};
    \node[] at (-.2,1.05) {\footnotesize{$S_1$}};
    \node[] at (-1.7,-.25) {\footnotesize{$S_2$}};
    \node[] at (1.3,-.25) {\footnotesize{$S_3$}};
    \node[] at (0,-2.5) {$F_{u_1}$};
    
    \filldraw[fill=red!9, draw=black, thick] (8,0) ellipse (3cm and 2cm); 
    \draw[->, >=stealth, thick] (8,2) arc (89.9:90.1:.01cm); 
    \filldraw[fill=white, draw=black, thick] (6.75,0) circle (.6cm); 
    \draw[->, >=stealth, thick] (6.77,0.6) -- (6.8,0.6);
    \filldraw[fill=white, draw=black, thick] (9.25,0) circle (.6cm);
    \draw[->, >=stealth, thick] (9.27,0.6) -- (9.3,0.6);
    \node[] at (5.5,1.5) {$u_2$};
    \node[] at (6.55,.25) {\footnotesize{$S_1$}};
    \node[] at (9.05,.25) {\footnotesize{$S_2$}};

    \node[] at (8,-2.5) {$F_{u_2}$};
\end{tikzpicture}

%% file: alexander.tex
\subsection{An efficient algorithm for embedded complexes} \label{SEC:efficientalgorithm}

In this section, we prove~\Cref{THM:embeddedeasy}.
To construct an efficient algorithm, we rely on \emph{Alexander duality}. Recall that, for an embedded complex $K \subseteq \R^n$, Alexander duality gives an isomorphism $H_{n-1}(K) \cong \Tilde{H}^0(\R^n \setminus \|K\|)$ between $(n-1)$-dimensional homology and $0$-dimensional (reduced) singular cohomology. For a subcomplex $L \subseteq K$, a class $\gamma \in H_{n-1}(K)$ lies in the image of $H_{n-1}(L) \to H_{n-1}(K)$ if and only if its Alexander dual lies in the image of the corresponding map $\Tilde{H}^0(\R^n \setminus \|L\|) \to \Tilde{H}^0(\R^n \setminus \|K\|)$. The latter map can be understood via an extension of the \emph{dual graph} of $K$.

\begin{definition}[Dual graph]
    Let $K$ be a simplicial complex. We denote by $\dualg{K}$ the \emph{dual graph} of~$K$, which has a vertex for each $n$-simplex of~$K$, and which has an edge between each pair of vertices whose corresponding simplices share an $(n-1)$-dimensional face.
\end{definition}

\begin{definition}[Extended dual graph]
For an embedded simplicial complex $K$, the extended dual graph $\dualg{K}^*$ of $K$ has one vertex for each $n$-simplex of $K$ and one vertex for each component of $\R^n \setminus \|K\|$. For each vertex $v$, we denote by $R_v \subseteq \R^n$ the interior of either the simplex or the connected component of $\R^n \setminus\|K\|$ to which $v$ corresponds. For each $(n-1)$-simplex $\tau$ of $K$, $\dualg{K}^*$ has an edge $(v,w)$ between the vertices for which $\tau \subseteq \overline{R_v} \cap \overline{R_w}$ (possibly allowing loops and multi edges). We denote the set of vertices corresponding to connected components of $\R^n \setminus \|\simp\| $ by $V_{\R^n \setminus K}$, and the vertex corresponding to the unbounded component by $v_\infty$. 
\end{definition}

\begin{figure}[h]
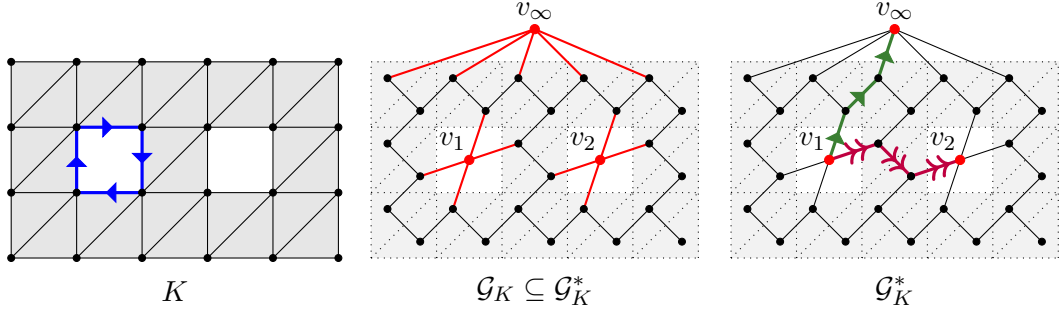

    \centering
    \include{tikz/extendeddual}
    \vspace{-1cm}
    \caption{Left: The embedded simplicial complex $K$ of \Cref{FIG:CUTEXAMPLE} with a class $\gamma \in H_1(K)$ (blue, single arrows).    
    Center: The (extended) dual graph of $K$. The vertices $V_{\R^n \setminus K} = \{ v_1, v_2, v_\infty\}$ and edges added to obtain $\dualg{K}^*$ from $\dualg{K}$ are in red.    
    Only some of the edges incident to $v_\infty$ are drawn. Right: A shortest path $v_1 \rightarrow v_\infty$ (green, single arrows) and a shortest path $v_1 \rightarrow v_2$ (purple, double arrows) in $\dualg{K}^*$. These correspond to the (minimum) $1$-cuts for $\gamma$ depicted in~\Cref{FIG:CUTEXAMPLE}, see~\Cref{PROP:cut_alex_dual_complete}. The paths are directed for visual clarity only.
    }
    \label{FIG:dualgraph}
\end{figure}

The outline for the rest of this section is as follows.
First, we show that one can associate a subgraph $\dualg{C}$ of $\dualg{K}^*$ to any $C \subseteq K^{(n-1)}$ so that $H_{n-1}(K - C) \cong \Tilde{H}^0(\dualg{C})$ (and this association is functorial). In case $C =\emptyset$, $\mathcal{G}_C =\mathcal{G}_K$ and this comes down to Alexander duality, and we see that any $\gamma \in H_{n-1}(K)$ can be written as a sum of homology classes corresponding to vertices in $\vcomp{K}$. We then show that whether $C$ is a homology cut for $\gamma$ depends on the existence of paths in $\dualg{C}$ between specific vertices in $\vcomp{K}$. This allows us to reduce~\eqref{EQ:MHC} to a (number of) shortest path problems in $\dualg{K}^*$. See~\Cref{FIG:dualgraph} for an illustration.

\begin{lemma}\label{LEM:Alexander Duality}
    Let $K \subseteq \R^n$ be an embedded simplicial complex, and let $C \subseteq K^{(n-1)}$. Denote by $\dualg{C} \subseteq \dualg{K}^*$ the smallest subgraph of $\dualg{K}^*$ containing both $\vcomp{K}$ and the dual edges corresponding to the simplices in $C$. Then there is an isomorphism $H_{n-1}(K - C) \cong \Tilde{H}^0(\dualg{C})$. Moreover, if $C \subseteq C'$, then the following diagram commutes:
    \[
    \begin{tikzcd}
    H_{n-1}(K - C') \arrow[r, "\cong "] \arrow[d, " "'] & \Tilde{H}^0(\dualg{C'})\arrow[d, " "] \\
    H_{n-1}(K - C) \arrow[r, "\cong"] & \Tilde{H}^0(\dualg{C})
    \end{tikzcd}
    \]
\end{lemma}
\begin{proof}
    Deferred to~\Cref{SEC:omittedproofs:Alexander}.
\end{proof}
Recall that $H^0(\R^n \setminus \|K\|)$ consists of functions $f:\R^n \setminus \|K\| \to \R $ that are constant on each connected component of $\R^n \setminus \|K\|$. We represent $\Tilde{H}^0(\R^n \setminus \|K\|)$ by locally constant functions that are $0$ on the unbounded connected component. Similarly, we think of $\Tilde{H}^0(\dualg{C})$ as functions that are constant on the components of $\dualg{C}$, and $0$ on the component containing $v_\infty$. Moreover, if $C \subseteq C'$, the map $\Tilde{H}^0(\dualg{C'}) \to \Tilde{H}^0(\dualg{C})$ is given by a restriction of functions. The following proposition allows us to relate cuts for specific elements of $H_{n-1}(K)$ to paths in the extended dual graph $\dualg{K}^*$ of $K$.

\begin{proposition}\label{PROP:Cut_Alex_dual_easy}
    Let $v \in \vcomp{K} \setminus \{v_\infty\}$. Consider the map $\mathbf{1}_v \in \Tilde{H}^0(\dualg{\emptyset}) \cong  \Tilde{H}^0(\R^n \setminus \|K\|)$. A set $C \subseteq K^{(n-1)}$ is a homological cut for the Alexander dual $\mathrm{AD}(\mathbf{1}_v) \in H_{n-1}(K)$ iff the set of edges in $\dualg{K}^*$ dual to the simplices of $C$ contains a path from $v$ to another vertex in $\vcomp{K}$. (Here, $\mathrm{AD}(\mathbf{1}_v)$ is meant to denote the Alexander dual of $\mathbf{1}_{R_v} \in \Tilde{H}^0(\R^n \setminus \|K\|)$.)
\end{proposition}
\begin{proof}
    By Lemma \ref{LEM:Alexander Duality}, $\mathrm{AD}(\mathbf{1}_v)$ is in the image of $H_{n-1}(K \sminus C) \to H_{n-1}(K)$ if and only if $\mathbf{1}_v$ is in the image of $\Tilde{H}^0(\dualg{C}) \to  \Tilde{H}^0(\dualg{\emptyset})$. This is the case if and only if the function $\mathbf{1}_v$ extends to a function $\varphi \in \Tilde{H}^0(\dualg{C})$, which is possible if and only if $v$ is in a different component from any other vertex of $\vcomp{K}$ inside $\dualg{C}$. That means that $C$ is a homology cut for $\mathrm{AD}(\mathbf{1}_v)$ if and only if there is a path in $\dualg{C}$ from $v$ to some other vertex of $\vcomp{K}$, proving our claim. 
\end{proof}

We need the following generalization of the above proposition, that allows us to relate cuts for arbitrary elements of $H_1(K)$ to paths in $\dualg{K}^*$; see~\Cref{FIG:dualgraph} for an example.
\begin{proposition}\label{PROP:cut_alex_dual_complete}
    Let $\gamma = \sum_{v  \in \vcomp{K}  } \mathrm{AD}(\alpha_v \mathbf{1}_v) $ with $\alpha_{v_\infty} = 0$ be any element of $H_{n-1}(K)$. Then $C$ is a homology cut for $\gamma$ if and only if the set of edges dual to the simplices of $C$ contains a path between a pair of vertices $v,w \in \vcomp{K}$ for which $\alpha_v \neq \alpha_w$. 
\end{proposition}

\begin{proof}
    The proof is similar to that of Proposition~\ref{PROP:Cut_Alex_dual_easy}, with the following adjustment: The function $\psi = \sum_{v \in \vcomp{K}} \alpha_v \mathbf{1}_v \in \Tilde{H}^0(\dualg{\emptyset})$ extends to a function $\varphi \in \Tilde{H}^0(\dualg{C})$ iff no pair of vertices $v,w \in \vcomp{K}$ for which $\psi(v) = \alpha_v \neq \alpha_w = \psi(w)$ is in the same component of~$\dualg{C}$. 
\end{proof}

\begin{proof}[Proof of Theorem~\ref{THM:embeddedeasy}]
    To compute a minimum homological cut for a class $\gamma \in H_{n-1}(K)$ of an $n$-complex $K$ embedded in $\R^n$, we proceed in three steps: 
    \begin{enumerate}
        \item \textbf{Constructing the extended dual graph.} First, we construct~$\dualg{K}^*$. For this we use the Void Boundary Reconstruction algorithm~\cite[Section~4]{dual_graph_dey}, which outputs sets of oriented simplices $(\vec{\xi}_1,\ldots,\vec\xi_k)$ that constitute the boundaries of the components of $\R^n \setminus \|K\|$ (cf.~\Cref{FIG:dualgraph}) in time bounded by $O(|K|^2)$.
        We assume wlog that $\vec\xi_k$ constitutes the boundary of the unbounded component of $\R^n \setminus \|K\|$. The sum of oriented simplices in any of the $\vec\xi_i$ with $i<k$ bounding a component $R_i \subseteq \R^n \setminus \|K\|$, which we denote by $[\vec\xi_i]$, represents the Alexander dual to the function $f_i \in \Tilde{H}^0(\R^n \setminus \|K\|)$ that is constant $1$ on $R_i$, and $0$ everywhere else. 
        From the $(\vec{\xi}_1,\ldots,\vec\xi_k)$ we can construct $\dualg{K}^*$ in time linear in $k + |K|$. Each $\vec{\xi}_i$ corresponds to a vertex $v_i = v_{R_i} \in \vcomp{K} \subseteq V(\dualg{K}^*)$.
        \item \textbf{Expressing $\gamma$ in an appropriate basis.} We write $\gamma$ as $\gamma = \sum_{i=1} ^{k} \alpha_i[\vec\xi_i]$ with $\alpha_i \in \mathbb{F}$, $\alpha_k = 0$. This comes down to a change of basis, which can be done in matrix-multiplication time. The representation of $\gamma$ is precisely the representation as assumed in Proposition~\ref{PROP:cut_alex_dual_complete}, as $[\vec\xi_i]$ is the Alexander dual of~$\mathbf{1}_{v_i}$. 
        \item \textbf{Determining shortest paths.} Using breadth-first search, we can find the minimum distance between any two vertices $v_i,v_j \in \vcomp{K}$ inside $\dualg{K}^*$ for which $\alpha_i \neq \alpha_j$ in time bounded by $O(k \cdot |\dualg{K}^*|)$. Let $P$ denote a path in $\dualg{K}^*$ achieving this minimum. The set of $(n-1)$-simplices $C_P$ dual to the edges constituting $P$ form a cut for $\gamma$ by Proposition~\ref{PROP:cut_alex_dual_complete}. Since the edges dual to any cut $C$ for $\gamma$ must contain a path between vertices of $\vcomp{K}$ with differing coefficients, the cut $C_P$ is a minimum cut by minimality of $P$. \qedhere
    \end{enumerate}
\end{proof}

%% file: tikz/extendeddual.tex
\begin{tikzpicture}[scale=0.87]
\tikzset{
  doublemidarrow/.style={
    decoration={
      markings,
      mark=at position 0.8 with {\arrow{>>}}
    },
    postaction={decorate}
  }
}

\begin{scope}
\foreach \x in {0,1,2,3,4,5} {
    \foreach \y in {0,1,2,3} {
        \coordinate (v\x\y) at (\x,\y);
    }
}

\draw (2.5, -0.5) node {$K$};

\fill[gray!20] (v00) -- (v50) -- (v53) -- (v03) -- (v00);
\fill[white] (v11) -- (v12) -- (v22) -- (v21) -- (v11);
\fill[white] (v31) -- (v32) -- (v42) -- (v41) -- (v31);

\foreach \y in {0,1,2,3} {
    \draw (v0\y) -- (v5\y);
}
\foreach \x in {0,1,2,3,4,5} {
    \draw (v\x0) -- (v\x3);
}

\draw (v00) -- (v11);
\draw (v01) -- (v12);
\draw (v02) -- (v13);

\draw (v10) -- (v21);
\draw (v12) -- (v23);

\draw (v20) -- (v31);
\draw (v21) -- (v32);
\draw (v22) -- (v33);

\draw (v30) -- (v41);
\draw (v32) -- (v43);

\draw (v40) -- (v51);
\draw (v41) -- (v52);
\draw (v42) -- (v53);

\draw[blue, sloped, allow upside down, very thick] (v11) -- node {\midarrow} (v12);
\draw[blue, sloped, allow upside down, very thick] (v12) -- node {\midarrow} (v22);
\draw[blue, sloped, allow upside down, very thick] (v22) -- node {\midarrow} (v21);
\draw[blue, sloped, allow upside down, very thick] (v21) -- node {\midarrow} (v11);

\foreach \x in {0,1,2,3,4,5} {
    \foreach \y in {0,1,2,3} {
        \filldraw (v\x\y) circle (1.5pt);
    }
}

\end{scope}











\begin{scope}[xshift=5.5cm]
\foreach \x in {0,1,2,3,4} {
        \foreach \y in {0,1,2} {
            \coordinate (BR\x\y) at (\x + 0.75,\y +0.25);
            \coordinate (TL\x\y) at (\x + 0.25,\y +0.75);
        }
    }

\draw (2.5, -0.5) node {$\dualg{K} \subseteq \dualg{K}^*$};

\foreach \x in {0,1,2,3,4,5} {
    \foreach \y in {0,1,2,3} {
        \coordinate (v\x\y) at (\x,\y);
    }
}

\fill[gray!10] (v00) -- (v50) -- (v53) -- (v03) -- (v00);
\fill[white] (v11) -- (v12) -- (v22) -- (v21) -- (v11);
\fill[white] (v31) -- (v32) -- (v42) -- (v41) -- (v31);

\foreach \y in {0,1,2,3} {
    \draw[dotted] (v0\y) -- (v5\y);
}
\foreach \x in {0,1,2,3,4,5} {
    \draw[dotted] (v\x0) -- (v\x3);
}

\draw[dotted] (v00) -- (v11);
\draw[dotted] (v01) -- (v12);
\draw[dotted] (v02) -- (v13);

\draw[dotted] (v10) -- (v21);
\draw[dotted] (v12) -- (v23);

\draw[dotted] (v20) -- (v31);
\draw[dotted] (v21) -- (v32);
\draw[dotted] (v22) -- (v33);

\draw[dotted] (v30) -- (v41);
\draw[dotted] (v32) -- (v43);

\draw[dotted] (v40) -- (v51);
\draw[dotted] (v41) -- (v52);
\draw[dotted] (v42) -- (v53);

    \coordinate (LH) at (1.5, 1.5);
    \coordinate (RH) at (3.5, 1.5);
    \coordinate (infH) at (2.5, 3.5);

    \draw[red, thick] (infH) -- (TL02);
    \draw[red, thick] (infH) -- (TL12);    
    \draw[red, thick] (infH) -- (TL22);    
    \draw[red, thick] (infH) -- (TL32);  
    \draw[red, thick] (infH) -- (TL42);

    \draw[red, thick] (LH) -- (BR12);
    \draw[red, thick] (LH) -- (BR01);
    \draw[red, thick] (LH) -- (TL21);
    \draw[red, thick] (LH) -- (TL10);

    \draw[red, thick] (RH) -- (BR32);
    \draw[red, thick] (RH) -- (BR21);
    \draw[red, thick] (RH) -- (TL41);
    \draw[red, thick] (RH) -- (TL30);

    \filldraw[red] (LH) circle (2pt);
    \filldraw[red] (RH) circle (2pt);
    \filldraw[red] (infH) circle (2pt);

    \draw[above] (infH) node {$v_\infty$};
    \draw[above left] (LH) +(0.1,0) node {$v_1$};
    \draw[above left] (RH) +(0.1,0) node {$v_2$};

\foreach \x in {0,1,2,3,4} {
    \foreach \y in {0,2} {
        \filldraw (BR\x\y) circle (1.5pt);
        \filldraw (TL\x\y) circle (1.5pt);

        \draw (BR\x\y) -- (TL\x\y);
    }
}
\foreach \x in {0,2,4} {
    \foreach \y in {1} {
        \filldraw (BR\x\y) circle (1.5pt);
        \filldraw (TL\x\y) circle (1.5pt);
        \draw (BR\x\y) -- (TL\x\y);
    }
}
{
\draw (BR00) -- (TL10);
\draw (BR10) -- (TL20);
\draw (BR20) -- (TL30);
\draw (BR30) -- (TL40);

\draw (BR02) -- (TL12);
\draw (BR12) -- (TL22);
\draw (BR22) -- (TL32);
\draw (BR32) -- (TL42);

\draw (BR02) -- (TL01);
\draw (BR01) -- (TL00);

\draw (BR22) -- (TL21);
\draw (BR21) -- (TL20);

\draw (BR42) -- (TL41);
\draw (BR41) -- (TL40);
}

\end{scope}

\begin{scope}[xshift=11cm]
\foreach \x in {0,1,2,3,4} {
        \foreach \y in {0,1,2} {
            \coordinate (BR\x\y) at (\x + 0.75,\y +0.25);
            \coordinate (TL\x\y) at (\x + 0.25,\y +0.75);
        }
    }

\draw (2.5, -0.5) node {$\dualg{K}^*$};

\foreach \x in {0,1,2,3,4,5} {
    \foreach \y in {0,1,2,3} {
        \coordinate (v\x\y) at (\x,\y);
    }
}

\fill[gray!10] (v00) -- (v50) -- (v53) -- (v03) -- (v00);
\fill[white] (v11) -- (v12) -- (v22) -- (v21) -- (v11);
\fill[white] (v31) -- (v32) -- (v42) -- (v41) -- (v31);

\foreach \y in {0,1,2,3} {
    \draw[dotted] (v0\y) -- (v5\y);
}
\foreach \x in {0,1,2,3,4,5} {
    \draw[dotted] (v\x0) -- (v\x3);
}

\draw[dotted] (v00) -- (v11);
\draw[dotted] (v01) -- (v12);
\draw[dotted] (v02) -- (v13);

\draw[dotted] (v10) -- (v21);
\draw[dotted] (v12) -- (v23);

\draw[dotted] (v20) -- (v31);
\draw[dotted] (v21) -- (v32);
\draw[dotted] (v22) -- (v33);

\draw[dotted] (v30) -- (v41);
\draw[dotted] (v32) -- (v43);

\draw[dotted] (v40) -- (v51);
\draw[dotted] (v41) -- (v52);
\draw[dotted] (v42) -- (v53);

    \coordinate (LH) at (1.5, 1.5);
    \coordinate (RH) at (3.5, 1.5);
    \coordinate (infH) at (2.5, 3.5);

    \draw (infH) -- (TL02);
    \draw (infH) -- (TL12);    
    \draw (infH) -- (TL22);    
    \draw (infH) -- (TL32);  
    \draw (infH) -- (TL42);

    \draw (LH) -- (BR12);
    \draw (LH) -- (BR01);
    \draw (LH) -- (TL21);
    \draw (LH) -- (TL10);

    \draw (RH) -- (BR32);
    \draw (RH) -- (BR21);
    \draw (RH) -- (TL41);
    \draw (RH) -- (TL30);

    \draw[above] (infH) node {$v_\infty$};
    \draw[above left] (LH) +(0.1,0) node {$v_1$};
    \draw[above left] (RH) +(0.1,0) node {$v_2$};

{
\draw (BR00) -- (TL10);
\draw (BR10) -- (TL20);
\draw (BR20) -- (TL30);
\draw (BR30) -- (TL40);

\draw (BR02) -- (TL12);
\draw (BR12) -- (TL22);
\draw (BR22) -- (TL32);
\draw (BR32) -- (TL42);

\draw (BR02) -- (TL01);
\draw (BR01) -- (TL00);

\draw (BR22) -- (TL21);
\draw (BR21) -- (TL20);

\draw (BR42) -- (TL41);
\draw (BR41) -- (TL40);
}

    \draw[OliveGreen, sloped, allow upside down, very thick] (LH) -- node {\midarrow} (BR12);
    \draw[OliveGreen, sloped, allow upside down, very thick] (BR12) -- node {\midarrow} (TL22);
    \draw[OliveGreen, sloped, allow upside down, very thick] (TL22) -- node {\midarrow} (infH);

    \draw[purple, doublemidarrow, very thick] (LH) -- (TL21);

    \draw[purple, doublemidarrow, very thick] (BR21) -- (RH);

    \filldraw[red] (LH) circle (2pt);
    \filldraw[red] (RH) circle (2pt);
    \filldraw[red] (infH) circle (2pt);

\foreach \x in {0,1,2,3,4} {
    \foreach \y in {0,2} {
        \filldraw (BR\x\y) circle (1.5pt);
        \filldraw (TL\x\y) circle (1.5pt);

        \draw (BR\x\y) -- (TL\x\y);
    }
}
\foreach \x in {0,2,4} {
    \foreach \y in {1} {
        \filldraw (BR\x\y) circle (1.5pt);
        \filldraw (TL\x\y) circle (1.5pt);
        \draw (BR\x\y) -- (TL\x\y);
    }
}

    \draw[purple, doublemidarrow, very thick] (TL21) -- (BR21);
    \filldraw (TL21) circle (1.5pt);
    \filldraw (BR21) circle (1.5pt);

\end{scope}

\end{tikzpicture}

%% file: HausdorffH.tex
\label{SEC:APPLICATION:HH}
Let $X = (X, d)$ be a (finite) metric space. 
For a subset $\sigma \subseteq X$, write $\mr(\sigma)$ for its diameter, i.e., the largest pairwise distance between points in $\sigma$. The Rips filtration of $X$ is typically defined as the $\R$-indexed family of simplicial complexes $\VR(X) := (\{ \sigma \subseteq X : \mr(\sigma) \leq r\})_{r \in \R}$. If we choose a total order $\preceq$ on the points in $X$, we can transform $\VR(X)$ into a discrete, simplex-wise filtration. Namely, we insert simplices one-by-one according to their diameter, breaking ties first by simplex dimension, and then by the lexicographic order on subsets of~$X$ induced by $\preceq$. We denote the resulting \emph{simplex-wise Rips filtration} by $\OVR(X)$. The \emph{length} of a bar $B = [\sigma_B, \tau_B)$ in the barcode of $\PH_p(\OVR(X))$ is $\ell(B) := \mr(\tau_B) - \mr(\sigma_B)$.

We consider the following heuristic to determine adversarial robustness of bars in the barcode of $\PH_p(\OVR(X))$ based on their length. We denote the Hausdorff distance between $A, B \subseteq X$ by $d_H(A, B) := \max \left\{ \sup_{a\in A}d(a,B), \sup_{b \in B}d(b,A) \right\}$.
\begin{definition}[Hausdorff heuristic] Let $X$ be a finite metric space. For $k \in \N$, we define 
    \begin{equation}
        \hhy := \max_{A \subseteq X, \, |A| \leq k} d_H(X \setminus A,X).
        \label{EQ:heuristic}
    \end{equation}
\end{definition}
The Hausdorff heuristic has the following two key properties. 
\begin{theorem}~\label{THM:main:HH}
    Let $p \geq 0$, and let $B$ be a bar in the barcode of $\PH_p(\OVR(X))$.
    If $B$ has length $\ell(B) \geq \hhy$, then $B$ is $k$-adversarially robust in degree $0$.
\end{theorem}

\begin{proposition} \label{PROP:HH:efficient}
    The parameter $\hhy$ can be computed in time $O(|X|^2 \log |X|)$.
\end{proposition}

\begin{proof}[Proof sketch of~\Cref{THM:main:HH}.] 
Consider the $\R$-indexed Rips filtration $\VR(X)$.
Let ${A \subseteq X}$, and set $\delta =d_H(X \setminus A,X)$. For any $r \in \R$,
there is a map $\mathrm{H}_p(\VR_r(X)) \to \mathrm{H}_p(\VR_{r + \delta} (X \sminus A))$ making the following diagram commute: 
\begin{center}    
\begin{tikzcd}[row sep=large]
H_p(\VR_{r}(X \setminus A))\arrow[r] \arrow[d]
& H_p(\VR_{r+\delta}(X \setminus A)) \arrow[d] \\
H_p(\VR_r(X)) \arrow[r] \arrow[ur]
& H_p(\VR_{r+\delta}(X)) 
\end{tikzcd}
\end{center}
\noindent This can be shown similarly to how stability of Rips persistence is proved, e.g., as in~\cite{Chazal2009}.
It follows that the cokernel of $\mathrm{H}_p(\VR(X \setminus A)) \to \mathrm{H}_p(\VR(X))$ is $\delta$-trivial, which would allow us to apply the induced matching theorem (\Cref{THM:inducedmatching}).

To transport this idea to the simplex-wise Rips filtration $\OVR(X)$, we make use of an ``$\varepsilon$-smoothening''. That is, an $\R$-indexed simplex-wise filtration $\epsVR(X)$, which inserts simplices in the same order as $\OVR(X)$, and which is $\varepsilon$-interleaved with $\VR(X)$, i.e.,
\begin{equation*} 
    \epsVR_r(X) \subseteq \VR_r(X) \subseteq \epsVR_{r+\epsilon}(X) \quad \forall r \in \R.
\end{equation*}
Such a filtration exists for any $\epsilon > 0$ small enough. Importantly, the persistence pairings and induced matchings of $\OVR(X)$ and $\epsVR(X)$ agree. 
By the previous, the cokernel of 
$\mathrm{H}_p(\epsVR(X \setminus A)) \to \mathrm{H}_p(\epsVR(X))$ is  $\delta+2\varepsilon$-trivial. By \Cref{THM:inducedmatching}, every bar of length greater than $\delta+2\varepsilon$ in the barcode of $\epsVR(X)$ is in the image of the matching induced by $\epsVR(X \setminus A) \hookrightarrow \epsVR(X)$. 
Making $\varepsilon$ arbitrarily small, the theorem follows. See~\Cref{SEC:omittedproofs:hausdorff} for details. 
\end{proof}

\begin{proof}[Proof of~\Cref{PROP:HH:efficient}]
To compute $\hhy$ efficiently, we use the following structural lemma on the sets $A$ that attain the maximum in~\eqref{EQ:heuristic}; its proof is deferred to~\Cref{SEC:omittedproofs:hausdorff}.
\begin{restatable}{lemma}{HHstructure} \label{LEM:HHstructure}
For $i \in \N$ and $x \in X$, let $\nu_i(x)$ denote the $ith$ nearest neighbor of $x$. The set $A \subseteq X$ attaining the maximum in~\eqref{EQ:heuristic} is of the form $\{x\} \cup \{ \nu_i(x) : i=1,\ldots,k\}$, with $x \in X$.
In particular, we have
    $
    \hhy= \max_{x \in X} d(x,\nu_{k}(x)).
    $
\end{restatable}
For each $x \in X$, we can compute $\nu_k(x)$ by sorting the points in $X$ by their distance to $x$, which can be done in $O(|X| \log |X|)$. This immediately proves~\Cref{PROP:HH:efficient}.
\end{proof}

%% file: LP.tex
\label{SEC:LP}
Let $K$ be a simplicial complex.
In this section, we derive and analyze a natural linear programming relaxation for the minimum homological $1$-cut problem for classes $\gamma \in H_1(K; \R)$. Throughout this section, we refer to $1$-cuts as \emph{edge cuts}. First, we recall some basic facts about cohomology.

We denote by $C^p(K;\mathbb{R}):= \mathrm{Hom}(C_p(K;\mathbb{R}),\mathbb{R})$ the space of cochains, and denote the coboundary operator by $\partial^*_p$, which is defined via $\left(\partial^*\varphi \right)(\sigma) = \varphi(\partial \sigma)$. The quotient of cocycles by coboundaries defines the cohomology $H^p(K;\mathbb{R})$ of $K$. A cohomology class $[\varphi] \in H^p(K;\mathbb{R})$ can be evaluated at a class $[c] \in H_p(K;\mathbb{R})$ via $[\varphi]([c]) = \varphi(c) $. This evaluation does not depend on the choice of representatives. Another way to phrase this is that a cochain $\varphi  \in C^p(K;\mathbb{R})$, when considered as a function on cycles $\varphi : Z_p(K;\mathbb{R}) \to \mathbb{R}$ is constant  on homology classes if and only if it satisfies the cocycle condition, i.e. $\partial^* \varphi = 0$. By a slight abuse of notation, we denote by $\mathrm{Supp}(\varphi)$ the \emph{simplices} (not chains) $\sigma$ for which $\varphi(\sigma) \neq 0$. 

\begin{lemma} \label{LEM:Cut supports cocycle}
	Suppose C is a minimal edge cut for  $\gamma=[c] \in H_1(K;\R)$. Then there exists a cocycle $\varphi \in C^1(K;\R)$ with $\mathrm{Supp}(\varphi) = C$ and $\varphi(c) = 1$.
\end{lemma}

\begin{proof}
	For $e\in C$, choose a cycle $c_e = \sum_i \lambda_{e_i} e_i \in C_1(K \sminus C \cup \{e\} ;\R)$ with $\lambda_e = 1$ and $[c_e] = \alpha_e \cdot [c] \neq 0$ for some $\alpha_e \in \R$, which is possible by minimality of $C$. For $e\notin C$ we set $c_e = 0$ and $\alpha_e = 0$. 
    Note that for any edge $e$, $c_e - e \in C_1(K \sminus C ;\R)$.
    Define $\varphi_C \in C_1(K; \R)$ via $\varphi_C(e) = \alpha_e$. 
    We claim that $\varphi_C$ satisfies the cocycle condition. 
    Let~$\sigma$ be any 2-simplex, with $\partial \sigma = e_1 - e_2 + e_3$. Consider the cycle $c_\sigma = c_{e_1} - c_{e_2} + c_{e_3} - \partial \sigma$. 
    We see that $c_\sigma \in C_1(K \sminus C ; \R)$, and that in $H_1(K ; \R)$, we have $[c_\sigma] = (\alpha_{e_1} - \alpha_{e_2} + \alpha_{e_3} ) [c] $. We conclude from the fact that $C$ is a cut that $\varphi_C(\partial \sigma) = \alpha_{e_1} - \alpha_{e_2} + \alpha_{e_3} = 0$.
    As $\varphi_C$ satisfies the cocycle condition, we have  $\varphi_C(c) = \varphi_C (\frac{1}{\alpha_e} c_e) =\frac{1}{\alpha_e} \varphi_C(e) = 1$ for any $e \in C$.
\end{proof}

This allows us to formulate the minimum homological edge cut problem as follows:

\begin{proposition}\label{PROP:ell0}
	Let $\|v\|_0$ denote the number of non-zero entries in a vector $v$. A minimum homological edge cut for a class $\gamma = [c] \in H_1(K; \R)$ is given by an optimum solution to the following optimization problem:
\begin{gather}
\tag{MC}
\label{EQ:mincutsparseLP}
\begin{aligned}
	\mathrm{mc} = \min_{\varphi \in C^1(K;\R) } \quad &\|\varphi\|_0\\
\textrm{s.t.} \quad & \partial^* \varphi &&= 0 \\
& \varphi(c) &&= 1
\end{aligned}
\end{gather}
(Here, $\varphi$ is interpreted as a vector via its decomposition $\varphi = \sum_{e} \lambda_e \mathbf{1}_{e}$ in the standard basis.)
\end{proposition}

\begin{proof}
	Note that if $\varphi$ is a cocycle with $\varphi(c) = 1$, then any cycle $c' = \sum \lambda_e e$  homologous to a scalar multiple of $c$ has $\lambda_e \neq 0$ for some $e \in \mathrm{Supp}(\varphi)$.
    It follows that $\mathrm{Supp}(\varphi)$ is a homological cut for $c$. Together with Lemma~\ref{LEM:Cut supports cocycle}, this proves the proposition. 
\end{proof}

\subsection{Linear programming relaxations}
Note that the constraints on $\varphi$ in~\eqref{EQ:mincutsparseLP} are linear. Thus, an optimum solution to~\eqref{EQ:mincutsparseLP} is the \emph{sparsest} solution to a set of linear (in)equalities. It is well-known that finding such a solution is NP-hard in general, see, e.g.,~\cite{hardnessofL0}. 
A common strategy to approximate problems of this form is to replace the non-convex $0$-norm by the convex $1$-norm~\cite{sparsetoLP:CompressiveSampling, sparsetoLP}.
In our case, this leads to:
\begin{gather}
\label{EQ:mincutLP} 
\tag{P}
\begin{aligned}
	\widetilde{\mathrm{mc}} = \min_{\varphi \in C^1(K;\R) } \quad &\|\varphi\|_1\\
\textrm{s.t.} \quad & \partial^* \varphi &&= 0 \\
& \varphi(c) &&= 1
\end{aligned}
\end{gather}
This optimization problem can be cast as a linear program (LP), and it can therefore be solved efficiently. Its dual is given by
\begin{gather}
\label{EQ:maxflowLP} 
\tag{D}
\begin{aligned}
	\mathrm{mf} = \max_{(B,r) \in C_2(K;\R) \times \R} \quad r \quad
\textrm{s.t.} \quad  \| \partial B + r \cdot c \|_\infty \leq 1.
\end{aligned}
\end{gather}

\subsection{Max-flow min-cut}
If $(B,r)$ is a feasible solution to~\eqref{EQ:maxflowLP}, then $\partial B + r \cdot c =  \sum_e \lambda_{e} e$ is a cycle homologous to $r \cdot c$ with $\lambda_{e} \leq 1$ for all edges $e$.
The LP thus computes a representative for the maximum multiple of $[c]$ which satisfies a capacity constraint of~$1$ on each edge. 
That is, it computes the \emph{maximum flow} in the subspace of $H_1(K;\R)$ generated by $[c]$. It is a homological analogue of the (classical) max-flow problem in graphs. Contrary to the situation for graphs, however, we do not have a \emph{max-flow min-cut theorem} in our setting. That is, we have $\mathrm{mc} \neq \mathrm{mf} \,(= \widetilde{\mathrm{mc}})$ in general, as the following example shows.

\begin{example}  \label{EXMP:mincutmaxflow} Consider the cycle $c$ drawn with red double arrows in the simplical complex depicted in~\Cref{fig:enter-label} (all coefficients are~$1$). Note that $c$ winds around the leftmost hole once, and winds around the rightmost hole twice. A minimum edge cut for $[c]$ consists of 3 edges, for example the three edges incident to the vertex $p$ that lie above $p$. However, the max flow for $[c]$ is only~$\frac{3}{2}$. Consider the cycles $c_1$ and $c_2$ drawn with single green arrows (all coefficients are~$1$). Note that~$c_1$ winds around the leftmost hole once, and $c_2$ winds around the rightmost hole once. Thus, $[c_1 + 2c_2] = [c]$.
Adding $\frac{1}{2}c_1 + c_2$ to $c$ yields a homology class  $[c+\frac{1}{2}c_1+c_2] = \frac{3}{2} [c]$, that is, achieving flow $\frac{3}{2}$. The cycle $c+\frac{1}{2}c + c_2$ fully uses the capacity constraints of $1$ on the three edges above $p$. Therefore, we cannot add another cycle that winds around the rightmost hole without violating the capacity constraint. This cycle is thus maximum. 
\end{example}
\begin{remark}
    The complex and homology class of~\Cref{EXMP:mincutmaxflow} satisfy the conditions of~\Cref{THM:embeddedeasy}. That is to say, the natural linear programming relaxation of~\eqref{EQ:MHC} is not tight, even in situations where an efficient algorithm exists.
\end{remark}

\begin{figure}[h]
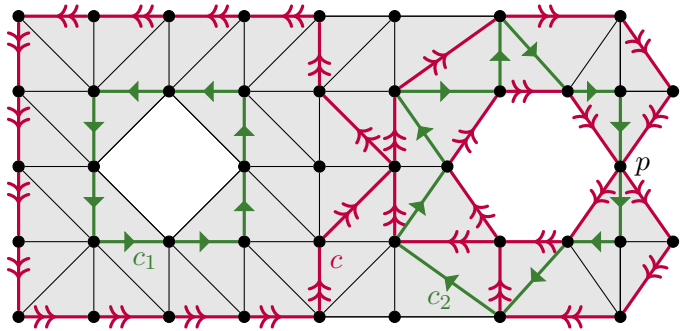

    \centering
\include{tikz/nomincutmaxflow}
\vspace{-1cm}
    \caption{The simplicial complex and cycles $c$ (red, double arrows) and $c_1, c_2$ (green, single arrows, on the left and right, respectively) of \Cref{EXMP:mincutmaxflow}.}
    \label{fig:enter-label}
\end{figure}

\subsection{Upper and lower bounds} Even though $\mathrm{mc} \neq \widetilde{\mathrm{mc}}$ in general, the linear program~\eqref{EQ:mincutLP} can potentially be used to compute lower and upper bounds on $\mathrm{mc}$.
First, note that the support of an optimal solution~$\varphi$ to \eqref{EQ:mincutLP} is an edge cut for~$[c]$, and so its cardinality is an upper bound on the size of a min-cut. 
Next, assuming that the cocycle $\varphi_C$ associated to a minimum edge cut $C$ for $[c]$ constructed in \Cref{LEM:Cut supports cocycle} satisfies $\| \varphi_C \|_1 
\leq \|\varphi_C\|_0$, we see that the optimum of~\eqref{EQ:mincutLP} gives a lower bound on the size of a min-cut. This assumption holds, in particular, when $\varphi_C(e) \leq 1$ for all edges $e \in C$. The latter seems like a reasonable condition; it is an interesting question whether we can characterize the complexes $K$ and classes $[c] \in H_1(K;\R) $ for which it holds.

%% file: tikz/nomincutmaxflow.tex
\begin{tikzpicture}[scale=1]

\tikzset{
  doublemidarrow/.style={
    decoration={
      markings,
      mark=at position 0.5 with {\arrow{>>}}
    },
    postaction={decorate}
  }
}

\draw[fill=gray!20] (-3,1) -- (-3, 5) -- (5,5) -- (5,1) -- cycle;
\draw[fill=gray!20] (5,5) -- (5.7, 4) -- (5,3) -- (5.7,2) -- (5, 1) -- cycle;
\draw[fill=white] (-2,3) -- (-1, 4) -- (0,3) -- (-1,2) -- cycle;
\draw[fill=white] (3.4,2) -- (2.7, 3) -- (3.4,4) -- (4.3,4) -- (5, 3) -- (4.3,2) -- cycle;

  \node[draw, circle, inner sep=1.5pt, fill] (A00) at (-3, 1) {};
\node[draw, circle, inner sep=1.5pt, fill] (A0) at (-2, 1) {};
    \node[draw, circle, inner sep=1.5pt, fill] (A1) at (-1, 1) {};
    \node[draw, circle, inner sep=1.5pt, fill] (A2) at (0, 1) {};
    \node[draw, circle, inner sep=1.5pt, fill] (A3) at (1, 1) {};
    \node[draw, circle, inner sep=1.5pt, fill] (A4) at (3.4, 1) {};
    \node[draw, circle, inner sep=1.5pt, fill] (A5) at (5, 1) {};
    \node[draw, circle, inner sep=1.5pt, fill] (Aextra) at (2, 1) {};

\node[draw, circle, inner sep=1.5pt, fill] (B00) at (-3, 2) {};
\node[draw, circle, inner sep=1.5pt, fill] (B0) at (-2, 2) {};
    \node[draw, circle, inner sep=1.5pt, fill] (B1) at (-1, 2) {};
    \node[draw, circle, inner sep=1.5pt, fill] (B2) at (0, 2) {};
    \node[draw, circle, inner sep=1.5pt, fill] (B3) at (1, 2) {};
    \node[draw, circle, inner sep=1.5pt, fill] (B4) at (2, 2) {};
    \node[draw, circle, inner sep=1.5pt, fill] (B5) at (3.4, 2) {};
    \node[draw, circle, inner sep=1.5pt, fill] (B6) at (4.3, 2) {};
    \node[draw, circle, inner sep=1.5pt, fill] (B7) at (5, 2) {};
    \node[draw, circle, inner sep=1.5pt, fill] (B8) at (5.7, 2) {};

\node[draw, circle, inner sep=1.5pt, fill] (C00) at (-3, 3) {};
\node[draw, circle, inner sep=1.5pt, fill] (C0) at (-2, 3) {};
    \node[draw, circle, inner sep=1.5pt, fill] (C2) at (0, 3) {};
    \node[draw, circle, inner sep=1.5pt, fill] (C3) at (1, 3) {};
    \node[draw, circle, inner sep=1.5pt, fill] (C4) at (2, 3) {};
    \node[draw, circle, inner sep=1.5pt, fill] (C5) at (2.7, 3) {};
    \node[draw, circle, inner sep=1.5pt, fill] (C6) at (5, 3) {};

\node[draw, circle, inner sep=1.5pt, fill] (D00) at (-3, 4) {};
\node[draw, circle, inner sep=1.5pt, fill] (D0) at (-2, 4) {};
    \node[draw, circle, inner sep=1.5pt, fill] (D1) at (-1, 4) {};
    \node[draw, circle, inner sep=1.5pt, fill] (D2) at (0, 4) {};
    \node[draw, circle, inner sep=1.5pt, fill] (D3) at (1, 4) {};
    \node[draw, circle, inner sep=1.5pt, fill] (D4) at (2, 4) {};
    \node[draw, circle, inner sep=1.5pt, fill] (D5) at (3.4, 4) {};
    \node[draw, circle, inner sep=1.5pt, fill] (D6) at (4.3, 4) {};
    \node[draw, circle, inner sep=1.5pt, fill] (D7) at (5.7, 4) {};
    \node[draw, circle, inner sep=1.5pt, fill] (D7extra) at (5, 4) {};

\node[draw, circle, inner sep=1.5pt, fill] (E00) at (-3, 5) {};
\node[draw, circle, inner sep=1.5pt, fill] (E0) at (-2, 5) {};
    \node[draw, circle, inner sep=1.5pt, fill] (E1) at (-1, 5) {};
    \node[draw, circle, inner sep=1.5pt, fill] (E2) at (0, 5) {};
    \node[draw, circle, inner sep=1.5pt, fill] (E3) at (1, 5) {};
    \node[draw, circle, inner sep=1.5pt, fill] (E4) at (2, 5) {};
    \node[draw, circle, inner sep=1.5pt, fill] (E5) at (3.4, 5) {};
    \node[draw, circle, inner sep=1.5pt, fill] (E6) at (5, 5) {};
    
    \foreach \i/\j in {A1/A2, A2/A3, A3/A4, A4/A5, 
                       A1/B1, A2/B2, A3/B3, A4/B4, A3/B4, 
                       B1/B2, B2/B3, B3/B4, B4/B5, B5/B6, B6/B8,
                       B2/C2, B3/C3, B4/C4, B5/C5, B6/C6, 
                       C2/C3, C3/C4, C4/C5,
                       C2/D2, C3/D3, C4/D4, C5/D5, C6/D6, D6/D7,
                       D1/D2, D2/D3, D3/D4, D4/D5, D5/D6, D6/D7,
                       D1/E1, D2/E2, D3/E3, D4/E4, D5/E5, D6/E6,
                       E1/E2, E2/E3, E3/E4, E4/E5, E5/E6,
                       A5/B6, A5/B8, B7/C6, B7/A5,
                       E6/D7, D7/C6, C6/B8, D7extra/E6, D7extra/C6,
                       A2/B1, C2/B1, C2/D1, E2/D1,
                       A3/B2, B3/C2, C3/D2, E3/D2,
                              B3/C4, C4/D3, E4/D3,
                       A4/B5, A4/B6,
                       B4/C5, C5/D4,
                       D4/E5, E5/D6,
                       B0/B00,B0/B1,B0/C0,B0/A0, B0/A00, B0/C00,B0/A1,
                       C0/C00,C0/D0, C0/D00, C0/D1, C0/B1,
                       D0/D00, D0/E00, D0/E0, D0/E1,D0/D1,
                       Aextra/B4} {
        \draw[-] (\i) -- (\j);
        }
        \foreach \i/\j in { A3/B3, B3/C4, C4/D4, D4/E5, E5/E6, E6/D7, D7/C6, C6/B8, B8/A5, A5/A4, A4/B5, B5/C5, C5/D5, D5/D6, D6/C6, C6/B6, B6/B5, B5/B4, B4/C4, C4/D3, D3/E3,
        E3/E2, E2/E1, E1/E0, E0/E00,
        E00/D00, D00/C00, C00/B00, B00/A00,
        A00/A0, A0/A1, A1/A2, A2/A3} {
        \draw[purple, sloped, allow upside down, very thick, doublemidarrow] (\i) -- node {} (\j);
        }
        
        \node[xshift=0.3cm] () at (C6) {$p$};
        \node[xshift=-0.32cm, yshift=-0.25cm, OliveGreen] () at (B1) {$c_1$};
        \node[xshift=-0.8cm, yshift=0.22cm, OliveGreen] () at (A4) {$c_2$};
        \node[xshift=0.22cm, yshift=-0.27cm, purple] () at (B3) {$c$};

        \foreach \i/\j in {D7extra/C6, C6/B7, B7/B6, B6/A4, A4/B4, B4/C5,C5/D4, D4/D5, D5/E5, E5/D6, D6/D7extra} {
        \draw[OliveGreen, sloped, allow upside down, very thick] (\i) -- node {\midarrow} (\j);
        }

        \foreach \i/\j in {D0/D1, D1/D2, D2/C2, C2/B2, B2/B1, B1/B0, B0/C0, C0/D0} {
        \draw[OliveGreen, sloped, allow upside down, very thick] (\j) -- node {\midarrow} (\i);
        }
        
\end{tikzpicture}

%% file: Appendix.tex
\label{SEC:omittedproofs}

\subsection{\texorpdfstring{Proofs omitted from~\Cref{SEC:NPhardness}}{}}
\label{SEC:omittedproofs:hardness}

\begin{proof}[Proof of~\Cref{LEM:Triangulation}]
    We describe the triangulation $L_u$ of the surface $F_u$ associated to an element $u \in U$. There are two basic ``building blocks" we use in our construction. The first consists of 4 triangles, triangulating a disk with 3 holes. The second is a ``double collar", consisting of 12 triangles, that will be attached to boundary components of the first building block. These building blocks are depicted in the left and center positions of Figure~\ref{fig:triangulation_building_blocks}. Note that a minimum 1-cut for the 1-dimensional homology class generated by the outer boundary of the first building block consists of 2 edges. Also note that a minimum 1-cut for the 1-dimensional homology class generated by the outer boundary of the double collar consists of 5 edges. Moreover, such a cut can be chosen to contain any of the edges on its outer boundary, and any of the edges on its interior boundary.
    
    The rightmost complex in Figure~\ref{fig:triangulation_building_blocks} shows how the collar is attached to the outer boundary of the first building block. This is the main building block we use, which we call $M$. A minimum $1$-cut for the homology class generated by its outer boundary consists of 6 edges, and can be chosen to contain any one of the edges on its outer boundary component. 
\begin{figure}[h]
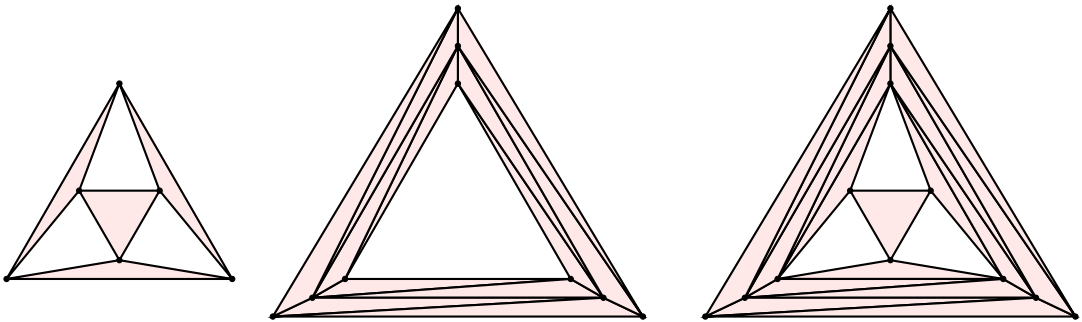

    \centering
\include{tikz/collar_triangulation}
\vspace{-1cm}
    \caption{On the left, the triangulation of the basic building block triangulating a disk with 3 holes. In the middle the triangulation of the double collar. On the right the main building block $M$, defined by gluing of the double collar along its interior boundary to the outer boundary of the disk with three holes.}
    \label{fig:triangulation_building_blocks}
\end{figure}

We now construct $L_u$ as follows:
\begin{enumerate}
    \item We start with a copy of $M$, setting $L_u' =M$, consisting 12 triangles. In Figure~\ref{fig:triangulation_L_u} we see on the left what $L_u'$  looks like at the start.
    \item While the number of interior boundary components of $L_u'$ is (strictly) less than $k$: for each interior boundary component we take a copy of $M$ and glue it along its outer boundary to the interior boundary component of $L_u'$.\\
    Note that this step triples the amount of boundary components of $L_u'$ every time it is applied. In Figure~\ref{fig:triangulation_L_u} we see on the right what $L_u'$ looks like after applying step (2) once. If $\lceil \log_3(k)\rceil>9$, it needs to be applied again. At the next iteration it will have 27 interior boundary components. 
    \item Once the number of interior boundary components exceeds $\lceil \log_3(k)\rceil$, for each element of $\mathcal{S}$ containing $u$, we glue a copy of the double collar to an interior boundary component.
    \item We label the outer boundary by $u$, and the interior boundary components of the collars attached in the previous step by the corresponding subsets of $\mathcal{S}$ in which $u$ occurs. Unused boundary components we remove by gluing in a triangle. The simplicial complex obtained from this is $L_u$.
\end{enumerate}
\begin{figure}[h]
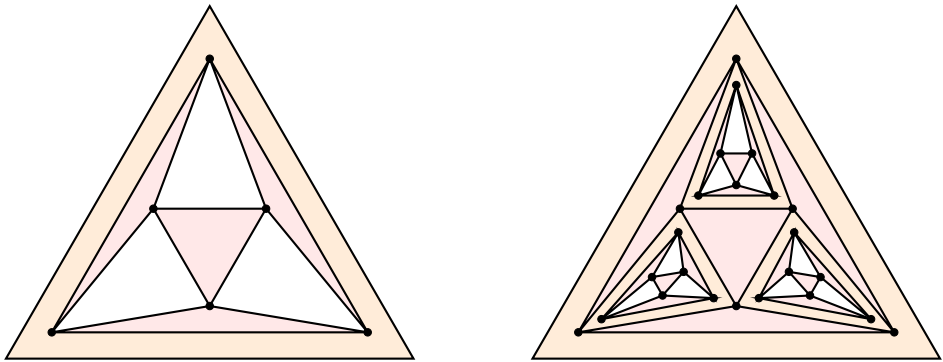

    \centering
\include{tikz/triangulation_L_u}
\vspace{-1cm}
    \caption{On the left $L_u'$ at step (1). This is the same as the main building block $M$, but we have chosen to represent the collar with a different colour, and without drawing all of its edges. On the right we see what $L_u'$ looks like after one iteration of step (2). }
    \label{fig:triangulation_L_u}
\end{figure}
In Figure~\ref{fig:axample_triangulation_L_u} we show an example of $L_u$ for an element $u$ contained in 4 sets, $S_1,S_2,S_3,S_4$, when $\lceil \log_3(k) \rceil=2$. We obtain $K(U,\mathcal{S})$ by gluing together all the triangulations $L_u$ along the identically labeled boundary components, and by identifying all boundary components labeled by elements of~$U$.
\begin{figure}[h]
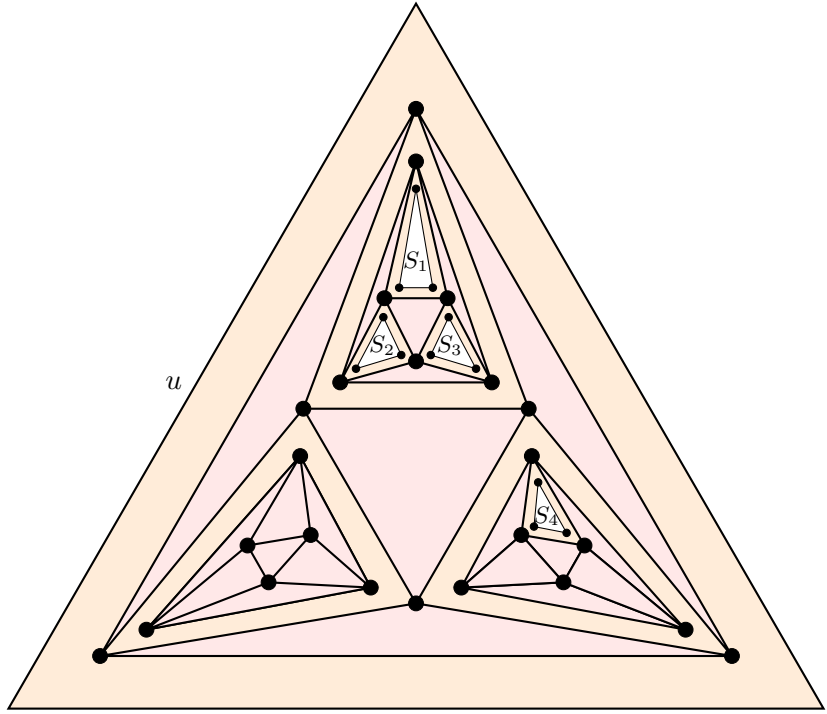

    \centering
\include{tikz/triangulation_example}
\vspace{-1cm}
    \caption{The simplicial complex $L_u$ for a $u$ occurring in sets $S_1,S_2,S_3,S_4$, and where $k= 2$.}
    \label{fig:axample_triangulation_L_u}
\end{figure}
We now verify the properties we claim $L_u$ to have.
\begin{enumerate}
    \item At step (1) $L_u'$ consists of 16 triangles. Applying step (2) once adds $3\cdot 16$ triangles, for a total of 64. Applying it again adds another $3^2 \cdot 16$ triangles, for a total of 208. More generally, after $\ell$ iterations of step (2), it $L_u'$ consists of $8(3^\ell -1)$ triangles. Step (2) is iterated $\lceil \log_3(k)\rceil -1$ times, so the total number of triangles of $L_u'$ at the end of (2) is in $O(3^{\lceil \log_3(k)\rceil })  =O(k)$, which is bounded by $O(|\mathcal{S}|)$. Step $(3)$ adds at most $12 \cdot k  $ triangles, which is bounded by $O(\log|\mathcal{S}|)$. Finally, the number of triangles added at step  (4) is at most $k$, which is also bounded by $O(|\mathcal{S}|)$. It follows that the number of triangles in $L_u$ is in $O(|\mathcal{S}|)$. We have such a complex for each $u \in U$, so the number of triangles in $K(U,\mathcal{S})$ is in $O(|U| \cdot |\mathcal{S}|)$.
    \item At step (1) a minimum homology cut for the homology class represented by the outer boundary of $L_u'$ consists of 6 edges, 2 on a boundary component, and 4 internal edges. It can be chosen to contain any edge of the outer boundary and an edge of any of the interior boundary components. For a thorough proof of this claim, we can apply Proposition~\ref{PROP:cut_alex_dual_complete}. After one iteration of step (2) this number has become $2\cdot 6 -1 = 11$, and each further iteration adds 5 edges. It remains true that the minimum cut can be chosen to contain precisely 1 edge from an interior boundary component. So at the end of step (2) the size of a minimum cut is $5\cdot \lceil \log_3(k)\rceil +1$. Adding the collars at step (3) adds an additional 4 edges, and step (4) adds nothing. So a minimum homology 1 cut for $[u] \in H_1(L_u)$ consists of $5\cdot \lceil \log_3(k)\rceil +5 = c+2$ edges.
    \item In the description of the main building block $M$ we already noted that a minimum homology cut for its outer boundary can be chosen to contain any of its outer boundary edges. As we iteratively glue copies of $M$ into itself, this remains true. As noted already above, the cut can be chosen to contain an edge from any of the interior boundary component. Since we attach a double collar to each of the ultimate interior boundary components, it is also true that we can choose the cut so that it contains \emph{any} of the edges in \emph{any} of the interior boundary components. \qedhere
\end{enumerate}

\end{proof}

\begin{proof}[Proof of Theorem~\ref{THM:NPhardness}]
    We show that there exists a solution to an instance $(U,\mathcal{S})$ of X3C if and only if there exists a 1-cut for $\gamma  \in H_1(K(U,\mathcal{S}))$ consisting of at most $c \cdot |U| + \frac{|U|}{3} + 1$ edges, where $\gamma$ is the image of any $[u]\in H_1(L_u)$ under the identification map.  Before proceeding, let us remark the following three things: 
\begin{enumerate}
    \item There exists a solution to $(U,\mathcal{S})$ if and only if $\mathcal{S}$ contains a cover consisting of precisely $\frac{|U|}{3}$ elements. 
    \item If $C$ is a cut for $\gamma$, then for each $u\in U$, the intersection of $C$ with the edges of $L_u$ yields a cut for $[u] \in H_1(L_u)$.  Indeed, suppose to the contrary that $[u] \in \mathrm{Im}(H_1(L_u \sminus (C \cap L_u))  \to H_1(L_u)) $, then by commutativity of the following diagram, $\gamma \in \mathrm{Im}(H_1(K\sminus C) \to H_1(K) )$:
    \[
    \begin{tikzcd}
    H_{1}(L_u \sminus (C\cap L_u) ) \arrow[r, " "] \arrow[d, " "'] & H_1(L_u)\ni [u]  \arrow[d, " "] \arrow[d,mapsto ,shift left =9]  \\
    H_{1}(K - C) \arrow[r, ""] & H_1(K)  \ni \gamma
    \end{tikzcd}
    \]
    \item Suppose $C_u$ is a minimum cut for $[u] \in H_1(L_u)$, containing the edge $e_u$ in the outer boundary component $u$, and the edge $e_{S_C}$ in one of the interior boundary components, corresponding to a set $S_C \in \mathcal{S}_u$. Then each 2-simplex of $L_u$ contains precisely 0 or 2 elements of $C$ in its boundary. It follows that there exists a 1-cochain $\varphi_u \in C^1(L_u)$ with $\varphi_u(\vec{e}_u) = 1$ and $\varphi_u( \vec e_{S_u}) = -1$, taking   values $\pm 1$ on the remaining edges of $C$, and $0$ everywhere else, satisfying $\partial^* \varphi_u = 0$. In other words, $\varphi_u$ is a cocycle, and so for each cycle $\zeta \in C_1(L_u)$ homologous to $u$, $\varphi_u (\zeta) = 1$. 
\end{enumerate}

Now suppose there is a solution $\mathcal{S}^* \subseteq \mathcal{S}$ to $(U,\mathcal{S})$. For once and for all, choose an edge $e_\gamma$ occurring in $\gamma$, and for each $S \in \mathcal{S^*}$ an edge $e_S$ in the cycle $[S] \in H_1(K)$ corresponding to $S$. For each $u \in U$, denote the element of the solution containing $u$ by $S_u$, and choose a minimum cut $C_u$ for $[u] \in H_1(L_u)$ containing the edge corresponding to $e_\gamma$ in $u$ and the edge corresponding to $e_{S_u}$ in the boundary component corresponding to $S_u$. By point (2) of Lemma~\ref{LEM:Triangulation}, $C_u$ contains $c = c(U,\mathcal{S}) = 5 \lceil \log_3(k) \rceil +3$ internal edges and 2 boundary edges. Let $C$ denote the union of these cuts inside $K(U,\mathcal{S})$. As we have consistently chosen the same edges in unified boundary components, we see that $|C| = c \cdot |U| + \frac{|U|}{3} + 1$. Indeed, we have $c\cdot |U|$ internal edges, 1 edge for each element of the solution, of which there are $\frac{|U|}{3}$, and one in $\gamma$.\\
To see that $C$ is a cut for $\gamma$, let $\varphi_u \in C^1(L_u)$ denote the cocycle supported on $C_u$ described in remark (3) above. Considering each $L_u$ as a subcomplex of $K(U,\mathcal{S})$, we see that the functions $(\varphi_u)_{u \in U}$ agree on common intersections. It follows that they extend to a cochain $\varphi \in C^1(K(U,\mathcal{S})) $ with $\mathrm{supp}(\varphi) = C$. We see that $\partial^* \varphi = 0$, since each 2-simplex $\sigma$ of $K(U,\mathcal{S})$ occurs in precisely one subcomplex $L_u$, and $\varphi(\partial \sigma) = \varphi_u(\partial \sigma) = 0$. Hence, for each $\zeta \in C_1(K)$ homologous to $\gamma$,  $\varphi(\zeta) = \varphi(\gamma) = 1$, and so the support of $\varphi$ contains at least one edge occurring in $\zeta$. It follows that $C$ is indeed a cut for $\gamma.$

Conversely, suppose that $C$ is a cut for $\gamma$. Then we obtain a cover $\mathcal{C} \subseteq \mathcal{S}$ for $U$ as follows: By remark (2) above, $C$ contains a cut $C_u$ for each $[u] \in H_1(L_u)$, and so by property (3) of Lemma~\ref{LEM:Triangulation}, $C_u$ contains at least one edge from $u$, and one edge from an interior boundary component $S_{C,u}$ of $L_u$. The interior boundary components of $L_u$ are labeled by elements of $\mathcal{S}$ containing $u$, so $u\in S_{C,u}$.  It follows that $\mathcal{C} = \{S_{C,u} \, | \, u  \in U\} $ is a cover for $U$. Similarly to above, we see that $|C| \geq c \cdot |U| + |\mathcal{C}| +1 $. It follows that if $|C| = c \cdot |U| + \frac{|U|}{3} + 1$, then $|\mathcal{C}| = \frac{|U|}{3}$, and $\mathcal{C}$ is an exact cover.
This concludes the proof that there exists a solution to an instance $(U,\mathcal{S})$ of X3C if and only if there exists a 1-cut for $\gamma  \in H_1(K(U,\mathcal{S}))$ consisting of at most $c \cdot |U| + \frac{|U|}{3} + 1$ edges. This in turn completes the reduction from X3C. 
\end{proof}

\begin{remark}\label{REM:Hardness_vertex_cuts}
    To prove hardness for vertex cuts, the same triangulation $K(U,\mathcal{S})$ of Lemma~\ref{LEM:Triangulation} can be used. The entire proof is analogous to the one given for $s=1$. As an alternative for property (2) of Lemma~\ref{LEM:Triangulation} one would need to show that a minimum vertex cut for $[u] \in H_1(L_u)$ consists of $c' +2$ vertices, where $c' = 2 \lceil \log_3(k) \rceil +1$. Given a solution $\mathcal{S}^*$ to $(U,\mathcal{S})$, we can construct a cut vertex for $\gamma$ in a way analogous to the case of $s=1$. To show that this is an actual vertex cut, one can use the observation that $C$ is a vertex cut for $\gamma$ if and only if the collection of edges incident to the vertices of $C$ contains an edge cut for $\gamma$.  
\end{remark}

\subsection{\texorpdfstring{Proofs omitted from~\Cref{SEC:efficientalgorithm}}{}}
\label{SEC:omittedproofs:Alexander}
\begin{proof}[Proof of~\Cref{LEM:Alexander Duality}]
    For each $C \subseteq K^{(n-1)}$, we construct a piecewise linear embedding $f_C: \dualg{C} \hookrightarrow \R^n \setminus \|K \sminus C\|$. We do so by constructing an embedding $f: \dualg{K}^* \hookrightarrow \R^n \setminus \|K \sminus K^{(n-1)}\|$, and restricting $f$ to $\dualg{C}$. 
    For each vertex $v$ of $\dualg{K}^*$, corresponding to a region $R_v$, choose a point $f(v)\in R_v$.
    Each edge $e = (v,w)$ in $\dualg{K}^*$ corresponds to an $(n-1)$-simplex $\sigma_e\subseteq\R^n$ constituting a common simplex in the boundaries of the regions $R_v$ and $R_w$ corresponding to $v$ and $w$ respectively.
    We let $f$ embed $e$ into $\mathring{R_v} \cup \mathring{R_w} \cup \sigma_e$ as a piecewise linear path between $f(v) \in R_v$ and $f(w)\in R_w$. From the observation that 
    \begin{align*}
        \R^n \setminus \|K \sminus C\| = \R^n \setminus \|K\| \cup \left(\bigcup_{v \in V(\dualg{C})} {R}_v  \right) \cup \left( \bigcup_{(v,w) \in E(\dualg{C})} \sigma_{(v,w)} \right)
    \end{align*}
    it readily follows that $f_C := f
    |_{\dualg{C}} $ embeds $\dualg{C}$ into $\R^n \setminus \|K \sminus C\|$, and induces a bijection on $\pi_0$, and hence an isomorphism on $\tilde{H}^0$. Now, for $C \subseteq C' \subseteq K^{(n-1)}$, consider the following diagram, where the leftmost horizontal isomorphisms are given by Alexander duality, and the vertical maps are induced by inclusion maps:
    \[
    \begin{tikzcd}
    H_{n-1}(K - C') \arrow[r, "\cong "] \arrow[d, " "'] & \Tilde{H}^0(\R^n \setminus \|K - C'\|)\arrow[d, " "] \arrow[r, "\cong ","f_{C'}^*"'] \arrow[d, " "'] & \Tilde{H}^0(\dualg{C'})\arrow[d, " "] \\
    H_{n-1}(K - C) \arrow[r, "\cong"] & \Tilde{H}^0(\R^n \setminus \|K - C\|) \arrow[r, "\cong", "f_C^*"'] & \Tilde{H}^0(\dualg{C})
    \end{tikzcd}
    \]
    The leftmost square commutes by functoriality of Alexander duality, and the rightmost square commutes by construction of $f_C$. The composition \[H_{n-1}(K \sminus C) \xrightarrow[]{A.D. } \Tilde{H}^0(\R^n \setminus (K - C)) \xrightarrow[]{f_C^*} \Tilde{H}^0(\dualg{C})\] gives the desired isomorphism between $H_{n-1}(K \sminus C)$ and $\Tilde{H}^0(\dualg{C})$, concluding the proof.
\end{proof}

\subsection{\texorpdfstring{Proofs omitted from~\Cref{SEC:APPLICATION:HH}}{}}
\label{SEC:omittedproofs:hausdorff}
\begin{proof}[Proof of \Cref{THM:main:HH}.]
The main idea is to use the induced matching theorem (\Cref{THM:inducedmatching}). To do so, we cannot work directly with the simplex-wise discretization $\OVR(X)$ of the Rips filtration. Indeed, \Cref{THM:main:HH} is concerned with the lengths of bars in the barcode of $\VR(X)$, which are completely forgotten after passing to $\OVR(X)$. Instead, we work with an \emph{$\epsilon$-smoothening} of~$\OVR(X)$, which is a simplex-wise filtration, indexed by $\R$, whose discretization is $\OVR(X)$, but which is $\epsilon$-interleaved with $\VR(X)$.

\begin{definition}
    Let $r_1<r_2<\ldots<r_M$ denote the critical values of the Rips filtration on~$X$, being the scales at which a simplex is added to the complex.
    Write $\delta(X) = \min_i (r_{i+1} - r_i)$, and let $\pi: (\mathcal{P}(X),~  \preceq) \to (\{1,\ldots,2^{|X|}\},~\leq)$ denote an order-preserving bijection on the simplices supported on $X$. 
    Choose $0<\epsilon < {\delta(X)}/{2}$, and write $\mathbf{f}(\sigma) := \mathbf{r}(\sigma) + \frac{\epsilon}{2^{|X|}} \cdot \pi(\sigma)$ for $\sigma \subseteq X$. 
    The $\epsilon$-smoothing $\epsVR(X)$ of $\OVR(X)$ is the $\R$-indexed filtration defined by $\epsVR_r(X) := \{\sigma \in \mathcal{P}(X) \, | \, \mathbf{f}(\sigma) \leq r\}$.
\end{definition}
From the definition, we see that $\epsVR(X)$ is simplex-wise, and simplices are inserted in the complex in the same order as $\OVR(X)$. Furthermore, we readily find that
\begin{equation} \label{EQ:Interleaved smoothing2}
    \epsVR_r(X) \subseteq \VR_r(X) \subseteq \epsVR_{r+\epsilon}(X) \quad \forall r \in \R.
\end{equation}
\begin{remark}\label{REM:Smooth matching}
    Bars in the barcode of $\OVR(X)$ are in one-to-one correspondence with bars in the barcode of $\epsVR(X)$ via the map $B = [\sigma_B, \tau_B) \mapsto [\mathbf{f}(\sigma_B),\mathbf{f}(\tau_B))$. 
    For $A \subseteq X$, the  matching $\imord_{X \setminus A \hookrightarrow X}$ induced by $\OVR(X) \setminus A \hookrightarrow \OVR(X)$ and the matching $\imatch^{\epsilon}_{X \setminus A \hookrightarrow X}$ induced by $\epsVR(X) \sminus A \hookrightarrow \epsVR(X)$ respect this identification in the sense that 
    \begin{align*}
        \imord_{X \setminus A \hookrightarrow X} \big([\sigma_{B'}, \tau_{B'})\big) = [\sigma_B, \tau_B) \, \textrm{ iff }\, \imatch^\epsilon_{X \setminus A \hookrightarrow X} \big([\mathbf{f}(\sigma_{B'}), \mathbf{f}(\tau_{B'}))\big) = [\mathbf{f}(\sigma_{B}), \mathbf{f}(\tau_{B})).
    \end{align*}
\end{remark}
Now, let $A \subseteq X$, and set $\delta = d_H(X \setminus A,X)$. For each $r \in \R$, there is a map $H_p(\VR_r(X)) \to H_p(\VR_{r+\delta}(X \setminus A))$ making the following diagram commute:
\begin{center}    
\begin{tikzcd}[row sep=large]
H_p(\VR_{r}(X \setminus A))\arrow[r] \arrow[d]
& H_p(\VR_{r+\delta}(X \setminus A)) \arrow[d] \\
H_p(\VR_r(X)) \arrow[r] \arrow[ur]
& H_p(\VR_{r+\delta}(X)) 
\end{tikzcd}
\end{center}
\noindent This can shown similarly to how stability of Rips persistence is proved, for example as in~\cite{Chazal2009}: The metric space $X$ can be embedded into $\ell^\infty (X)$ via $x \mapsto (y \mapsto d(x,y))$, and the Rips filtration on $X$ agrees with the \v{C}ech filtration on the image of $X$ in $\ell^\infty(X)$. By the Persistent Nerve Lemma \cite[Lemma 3.4]{chazal2008towards}, the \v{C}ech persistence on $X$ and $X \setminus A$ in $\ell^\infty(X)$ agree with the persistence of the sub levelset filtrations of the offset functions $d_X$ and $d_{X \setminus A}$ respectively. 
By definition, 
\[
d_X^{-1}(-\infty, r] \subseteq d_{X\setminus A}^{-1}(-\infty, r+d_H(X\setminus A,X)] = d_{X\setminus A}^{-1}(-\infty, r+\delta],
\]
proving our claim. 

From the commutativity of the diagram, it follows that the kernel and cokernel of $\PH(\VR(X \setminus A)) \to \PH(\VR(X))$ are $\delta$-trivial. From~\Cref{EQ:Interleaved smoothing2}, we may conclude that the cokernel of $\PH(\epsVR(X \setminus A)) \to \PH(\epsVR(X))$ is $(\delta+ 2\epsilon$)-trivial. 
Having chosen~$\epsilon > 0$ small enough, it then follows from the induced matching theorem (\Cref{THM:inducedmatching}) and~\Cref{REM:Smooth matching} that any bar $[\sigma_B,\tau_B)$ with $\mathbf{r}(\tau_B) - \mathbf{r}(\sigma_B) > \delta$ is in the image of the matching $\imord_{X\setminus A \hookrightarrow X}$ induced by $\OVR(X) \sminus A \hookrightarrow \OVR(X)$. 
We conclude that if $\mathbf{r}(\tau_B) - \mathbf{r}(\sigma_B) > \hhy$, then $[\sigma_B,\tau_B)$ is in the image of $\imord_{X \setminus A \hookrightarrow X}$ for any $A$ with $|A| \leq k$. In other words, $[\sigma_B,\tau_B)$ is $k$-robust. Note that this argument does not depend on the choice of ordering $\preceq$ (as this only affects the insertion order for simplices $\sigma, \tau$ with $\br(\sigma) = \br(\tau)$).
\end{proof}
\begin{proof}[Proof of \Cref{LEM:HHstructure}] 
Note that for any $A \subseteq X$, the Hausdorff distance $d_H( X \setminus A, X)$ equals
\begin{align*}
    \max \left\{\max_{x \in X\setminus A} d(x,X), ~\max_{x \in X}d(x, X\setminus A) \right\} 
    = \max_{x \in X}d(x, X\setminus A )
    = \max_{x \in A}d(x, X\setminus A ),   
\end{align*}
and so $\hhy = \max_{A \subseteq X, \, |A| = k} \max_{x \in A}d(x, X\setminus A )$.
Suppose $\max_{x \in X} d(x,\nu_k(x))$ is realized by $x_0 \in X$, and let $A_0 = \{x_0, \nu_1(x_0),\ldots, \nu_{k-1}(x_0)\}$. Then, we have 
\begin{align*}
   \hhy \geq d_H(X \setminus A_0,X) =  \max_{a \in A_0} d(a,X\setminus A_0) \geq d(x_0,\nu_k(x_0)). 
\end{align*}

Conversely, suppose $A_0 \subseteq X$ realizes $\max_{A \subseteq X,\, |A| \leq k} d_H(X\setminus A,X) = \hhy$. W.l.o.g., we may assume $|A|=k$. We claim that ${d_H(X\setminus A_0, X) = d(a, \nu_k(a))}$ for some $a \in A_0$. Suppose not. Let $a\in A_0$ and $x \in X$ such that $d_H(X \setminus A_0,X) = d(a, X \setminus A_0) = \min_{z \in X\setminus A_0}d(a,z) = d(a,x)$. Let $x'$ denote the second-nearest point to $a$ in $X \setminus A_0$ (only after $x$). By our assumption, $x$ is not the $k$-th nearest neighbor of $a$, so $x = \nu_j(a)$ for some $j < k$. There exists an $a' \in A_0$ with $d(a,a') > d(a,x)$ and $d(a', X\setminus A_0) < d(a, X \setminus A_0) = d(a,x)$. Now consider $A_0' := A_0 \cup \{x\} \setminus\{a'\}$. We see that 
\begin{align*}
    d(X \setminus A_0', X) \geq d(a, X \setminus A_0') =d(a,x') > d(a,x) = d(X \setminus A_0, X),
\end{align*}
contradicting our assumption that $A_0$ maximizes $d_H(X \setminus A, X)$. Thus, $d_H(X \setminus A_0, X) = d(a,\nu_k(a))$ for some $a \in A_0$, and so $\max_{A \subseteq X,\, |A| = k} d_H(X\setminus A,X) = d_H(X\setminus A_0,X) \leq \max_{x \in X} d(x,\nu_k(x))$.
\end{proof}

%% file: tikz/collar_triangulation.tex
\begin{tikzpicture}[scale=.5]
\def\x{5.2}
\def\y{11.5}

\filldraw[color = black, fill = red!9, thick] (0,0) -- (6,0) -- (3,.5) -- cycle;
\filldraw[color = black, fill = red!9, thick] (0,0) --  (3,5.2) -- (1.93,2.35) -- cycle;
\filldraw[color = black, fill = red!9, thick] (4.07,2.35) --  (3,\x) -- (6,0) -- cycle;
\filldraw[color = black, fill = red!9, thick] (1.93,2.35) --  (4.07,2.35) -- (3,.5) -- cycle;

\filldraw[black] (0,0) circle (2pt);
\filldraw[black] (6,0) circle (2pt);
\filldraw[black] (3,\x) circle (2pt);
\filldraw[black] (1.93,2.35) circle (2pt);
\filldraw[black] (4.07,2.35) circle (2pt);
\filldraw[black] (3,.5) circle (2pt);


\filldraw[color = black, fill = red!9, thick] (8.13,-0.5) -- (9,0) -- (15,0) -- cycle;

\filldraw[color = black, fill = red!9, thick] (8.13,-0.5) -- (15.87,-.5) -- (15,0) -- cycle;

\filldraw[color = black, fill = red!9, thick] (12,\x) -- (15.87,-.5) -- (15,0) -- cycle;
\filldraw[color = black, fill = red!9, thick] (12,\x) -- (15.87,-.5) -- (12,\x+1) -- cycle;

\filldraw[color = black, fill = red!9, thick] (12,\x) -- (9,0) -- (12,\x+1) -- cycle;
\filldraw[color = black, fill = red!9, thick] (8.13,-.5) -- (9,0) -- (12,\x+1) -- cycle;

\filldraw[color = black, fill = red!9, thick] (7.08,-1) -- (8.13,-0.5) -- (15.87,-0.5) -- cycle;

\filldraw[color = black, fill = red!9, thick] (7.08,-1) -- (16.92,-1) -- (15.87,-0.5) -- cycle;

\filldraw[color = black, fill = red!9, thick] (12,\x+1) -- (16.92,-1) -- (15.87,-0.5) -- cycle;
\filldraw[color = black, fill = red!9, thick] (12,\x+1) -- (16.92,-1) -- (12,\x+2) -- cycle;

\filldraw[color = black, fill = red!9, thick] (12,\x+1) -- (8.13,-0.5) -- (12,\x+2) -- cycle;
\filldraw[color = black, fill = red!9, thick] (7.08,-1) -- (8.13,-0.5) -- (12,\x+2) -- cycle;

\filldraw[black] (9,0) circle (2pt);
\filldraw[black] (15,0) circle (2pt);
\filldraw[black] (12,\x) circle (2pt);

\filldraw[black] (8.13,-0.5) circle (2pt);
\filldraw[black] (15.87,-.5) circle (2pt);
\filldraw[black] (12,\x + 1) circle (2pt);

\filldraw[black] (7.08,-1) circle (2pt);
\filldraw[black] (16.92,-1) circle (2pt);
\filldraw[black] (12,\x + 2) circle (2pt);


\filldraw[color = black, fill = red!9, thick] (9+\y,0) -- (15+\y,0) -- (12+\y,.5) -- cycle;
\filldraw[color = black, fill = red!9, thick] (9+\y,0) --  (12+\y,5.2) -- (10.93+\y,2.35) -- cycle;
\filldraw[color = black, fill = red!9, thick] (13.07+\y,2.35) --  (12+\y,\x) -- (15+\y,0) -- cycle;
\filldraw[color = black, fill = red!9, thick] (10.93+\y,2.35) --  (13.07+\y,2.35) -- (12+\y,.5) -- cycle;

\filldraw[black] (9+\y,0) circle (2pt);
\filldraw[black] (15+\y,0) circle (2pt);
\filldraw[black] (12+\y,\x) circle (2pt);
\filldraw[black] (10.93+\y,2.35) circle (2pt);
\filldraw[black] (13.07+\y,2.35) circle (2pt);
\filldraw[black] (12+\y,.5) circle (2pt);

\filldraw[color = black, fill = red!9, thick] (8.13+\y,-0.5) -- (9+\y,0) -- (15+\y,0) -- cycle;

\filldraw[color = black, fill = red!9, thick] (8.13+\y,-0.5) -- (15.87+\y,-.5) -- (15+\y,0) -- cycle;

\filldraw[color = black, fill = red!9, thick] (12+\y,\x) -- (15.87+\y,-.5) -- (15+\y,0) -- cycle;
\filldraw[color = black, fill = red!9, thick] (12+\y,\x) -- (15.87+\y,-.5) -- (12+\y,\x+1) -- cycle;

\filldraw[color = black, fill = red!9, thick] (12+\y,\x) -- (9+\y,0) -- (12+\y,\x+1) -- cycle;
\filldraw[color = black, fill = red!9, thick] (8.13+\y,-.5) -- (9+\y,0) -- (12+\y,\x+1) -- cycle;

\filldraw[color = black, fill = red!9, thick] (7.08+\y,-1) -- (8.13+\y,-0.5) -- (15.87+\y,-0.5) -- cycle;

\filldraw[color = black, fill = red!9, thick] (7.08+\y,-1) -- (16.92+\y,-1) -- (15.87+\y,-0.5) -- cycle;

\filldraw[color = black, fill = red!9, thick] (12+\y,\x+1) -- (16.92+\y,-1) -- (15.87+\y,-0.5) -- cycle;
\filldraw[color = black, fill = red!9, thick] (12+\y,\x+1) -- (16.92+\y,-1) -- (12+\y,\x+2) -- cycle;

\filldraw[color = black, fill = red!9, thick] (12+\y,\x+1) -- (8.13+\y,-0.5) -- (12+\y,\x+2) -- cycle;
\filldraw[color = black, fill = red!9, thick] (7.08+\y,-1) -- (8.13+\y,-0.5) -- (12+\y,\x+2) -- cycle;

\filldraw[black] (9+\y,0) circle (2pt);
\filldraw[black] (15+\y,0) circle (2pt);
\filldraw[black] (12+\y,\x) circle (2pt);

\filldraw[black] (8.13+\y,-0.5) circle (2pt);
\filldraw[black] (15.87+\y,-.5) circle (2pt);
\filldraw[black] (12+\y,\x + 1) circle (2pt);

\filldraw[black] (7.08+\y,-1) circle (2pt);
\filldraw[black] (16.92+\y,-1) circle (2pt);
\filldraw[black] (12+\y,\x + 2) circle (2pt);
\end{tikzpicture}

%% file: tikz/triangulation_L_u.tex
\begin{tikzpicture}[scale=.7]
\def\x{5.2}

\filldraw[draw=none, fill = orange!15, thick]
  (-0.87,-0.5) -- (3,\x+1) -- (6.87,-0.5) -- (-0.87,-0.5) -- (0,0) --
  (6,0) -- (3,\x) -- (0,0) --
  cycle;

\draw[thick]   (-0.87,-0.5) -- (3,\x+1) -- (6.87,-0.5) -- (-0.87,-0.5) -- cycle;


\filldraw[color = black, fill = red!9, thick] (0,0) -- (6,0) -- (3,.5) -- cycle;
\filldraw[color = black, fill = red!9, thick] (0,0) --  (3,5.2) -- (1.93,2.35) -- cycle;
\filldraw[color = black, fill = red!9, thick] (4.07,2.35) --  (3,\x) -- (6,0) -- cycle;
\filldraw[color = black, fill = red!9, thick] (1.93,2.35) --  (4.07,2.35) -- (3,.5) -- cycle;

\filldraw[black] (0,0) circle (2pt);
\filldraw[black] (6,0) circle (2pt);
\filldraw[black] (3,\x) circle (2pt);
\filldraw[black] (1.93,2.35) circle (2pt);
\filldraw[black] (4.07,2.35) circle (2pt);
\filldraw[black] (3,.5) circle (2pt);






\filldraw[draw=none, fill = orange!15, thick]
  (9.13,-0.5) -- (13,\x+1) -- (16.87,-0.5) -- (9.13,-0.5) -- (10,0) --
  (16,0) -- (13,\x) -- (10,0) --
  cycle;

\filldraw[draw=none, fill = orange!15, thick]
  (10,0) -- (11.93,2.35) -- (13,.5) -- (10,0) -- (10.44,0.25) -- (12.57,0.65) -- (11.9,1.9) -- (10.44,0.25) --
  cycle;

\filldraw[draw=none, fill = orange!15, thick]
  (16,0) -- (14.07,2.35) -- (13,.5) -- (16,0) -- (15.56,0.25) -- (13.43,0.65) -- (14.1,1.9) -- (15.56,0.25) --
  cycle;

\filldraw[draw=none, fill = orange!15, thick]
  (11.93,2.35) -- (13,\x) -- (14.07,2.35) -- (11.93,2.35) -- (12.28,2.6) -- (13.72,2.6) -- (13,\x-.5) -- (12.28,2.6) --
  cycle;

\draw[thick] (9.13,-0.5) -- (13,\x+1) -- (16.87,-0.5) -- (9.13,-0.5) -- cycle;

\filldraw[color = black, fill = red!9, thick] (10,0) -- (16,0) -- (13,.5) -- cycle;
\filldraw[color = black, fill = red!9, thick] (10,0) -- (13,5.2) -- (11.93,2.35) -- cycle;
\filldraw[color = black, fill = red!9, thick] (14.07,2.35) -- (13,\x) -- (16,0) -- cycle;
\filldraw[color = black, fill = red!9, thick] (11.93,2.35) -- (14.07,2.35) -- (13,.5) -- cycle;

\filldraw[color = black, fill = red!9, thick] (13,2.8) -- (12.7,3.4) -- (13.3,3.4) -- cycle;
\filldraw[color = black, fill = red!9, thick] (13,2.8) -- (12.28,2.6) -- (13.72,2.6) -- cycle;
\filldraw[color = black, fill = red!9, thick] (13,\x-.5) -- (12.7,3.4) -- (12.28,2.6) -- cycle;
\filldraw[color = black, fill = red!9, thick] (13,\x-.5) -- (13.3,3.4) -- (13.72,2.6) -- cycle;

\filldraw[color = black, fill = red!9, thick] (11.6,.7) -- (12.0,1.15) -- (11.4,1.05) -- cycle;
\filldraw[color = black, fill = red!9, thick] (11.9,1.9) -- (12.0,1.15) -- (12.57,0.65) -- cycle;
\filldraw[color = black, fill = red!9, thick] (11.9,1.9) -- (10.44,.25) -- (11.4,1.05) -- cycle;
\filldraw[color = black, fill = red!9, thick] (12.57,0.65) -- (10.44,.25) -- (11.6,.7) -- cycle;

\filldraw[color = black, fill = red!9, thick] (14.4,.7) -- (14.0,1.15) -- (14.6,1.05) -- cycle;
\filldraw[color = black, fill = red!9, thick] (14.1,1.9) -- (14.0,1.15) -- (13.43,0.65) -- cycle;
\filldraw[color = black, fill = red!9, thick] (14.1,1.9) -- (15.56,.25) -- (14.6,1.05) -- cycle;
\filldraw[color = black, fill = red!9, thick] (13.43,0.65) -- (15.56,.25) -- (14.4,.7) -- cycle;

\filldraw[black] (10,0) circle (2pt);
\filldraw[black] (16,0) circle (2pt);
\filldraw[black] (13,\x) circle (2pt);

\filldraw[black] (11.93,2.35) circle (2pt);
\filldraw[black] (14.07,2.35) circle (2pt);
\filldraw[black] (13,.5) circle (2pt);

\filldraw[black] (13,2.8) circle (2pt);
\filldraw[black] (12.7,3.4) circle (2pt);
\filldraw[black] (13.3,3.4) circle (2pt);

\filldraw[black] (12.0,1.15) circle (2pt);
\filldraw[black] (11.6,.7) circle (2pt);
\filldraw[black] (11.4,1.05) circle (2pt);

\filldraw[black] (14.0,1.15) circle (2pt);
\filldraw[black] (14.4,.7) circle (2pt);
\filldraw[black] (14.6,1.05) circle (2pt);

\filldraw[black] (10.44,0.25) circle (2pt);
\filldraw[black] (11.9,1.9) circle (2pt);
\filldraw[black] (12.57,0.65) circle (2pt);

\filldraw[black] (15.56,0.25) circle (2pt);
\filldraw[black] (14.1,1.9) circle (2pt);
\filldraw[black] (13.43,0.65) circle (2pt);

\filldraw[black] (13,\x-0.5) circle (2pt);
\filldraw[black] (12.28,2.6) circle (2pt);
\filldraw[black] (13.72,2.6) circle (2pt);

%

\end{tikzpicture}

%% file: tikz/triangulation_example.tex
\begin{tikzpicture}[scale=1.4]
\def\x{5.2}

\filldraw[draw=none, fill = orange!15, thick]
  (9.13,-0.5) -- (13,\x+1) -- (16.87,-0.5) -- (9.13,-0.5) -- (10,0) --
  (16,0) -- (13,\x) -- (10,0) --
  cycle;

\filldraw[draw=none, fill = orange!15, thick]
  (10,0) -- (11.93,2.35) -- (13,.5) -- (10,0) -- (10.44,0.25) -- (12.57,0.65) -- (11.9,1.9) -- (10.44,0.25) --
  cycle;

\filldraw[draw=none, fill = orange!15, thick]
  (16,0) -- (14.07,2.35) -- (13,.5) -- (16,0) -- (15.56,0.25) -- (13.43,0.65) -- (14.1,1.9) -- (15.56,0.25) --
  cycle;

\filldraw[draw=none, fill = orange!15, thick]
  (11.93,2.35) -- (13,\x) -- (14.07,2.35) -- (11.93,2.35) -- (12.28,2.6) -- (13.72,2.6) -- (13,\x-.5) -- (12.28,2.6) --
  cycle;

\draw[thick] (9.13,-0.5) -- (13,\x+1) -- (16.87,-0.5) -- (9.13,-0.5) -- cycle;

\filldraw[color = black, fill = red!9, thick] (10,0) -- (16,0) -- (13,.5) -- cycle;
\filldraw[color = black, fill = red!9, thick] (10,0) -- (13,5.2) -- (11.93,2.35) -- cycle;
\filldraw[color = black, fill = red!9, thick] (14.07,2.35) -- (13,\x) -- (16,0) -- cycle;
\filldraw[color = black, fill = red!9, thick] (11.93,2.35) -- (14.07,2.35) -- (13,.5) -- cycle;

\filldraw[draw=none, fill = orange!15, thick]
  (12.7,3.4) -- (13.3,3.4) -- (13,\x-.5) -- (12.7,3.4) -- (12.84,3.5) -- (13, \x-.76) -- (13.16,3.5) -- (12.84,3.5) --
  cycle;

\draw[] (13, \x-.76) -- (13.16,3.5) -- (12.84,3.5) --
  cycle;


\filldraw[draw=none, fill = orange!15, thick]
 (13,2.8) -- (12.7,3.4) -- (12.28,2.6) -- (13,2.8) -- (12.86,2.86) -- (12.43,2.73) -- (12.69,3.22) -- (12.86,2.86) --
  cycle;

\draw[] (12.86,2.86) -- (12.69,3.22) -- (12.43,2.73) --
  cycle;

\filldraw[draw=none, fill = orange!15, thick]
  (13,2.8) -- (13.3,3.4) -- (13.72,2.6) -- (13,2.8) -- (13.14,2.86) -- (13.57,2.73) -- (13.31,3.22) -- (13.14,2.86) --
  cycle;

\draw[] (13.14,2.86) -- (13.31,3.22) -- (13.57,2.73) --
  cycle;

\filldraw[color = black, fill = red!9, thick] (13,2.8) -- (12.7,3.4) -- (13.3,3.4) -- cycle;
\filldraw[color = black, fill = red!9, thick] (13,2.8) -- (12.28,2.6) -- (13.72,2.6) -- cycle;
\filldraw[color = black, fill = red!9, thick] (13,\x-.5) -- (12.7,3.4) -- (12.28,2.6) -- cycle;
\filldraw[color = black, fill = red!9, thick] (13,\x-.5) -- (13.3,3.4) -- (13.72,2.6) -- cycle;

\filldraw[color = black, fill = red!9, thick] (11.9,1.9)  --  (10.44,.25) -- (12.57,0.65) -- cycle;
\filldraw[color = black, fill = red!9, thick] (11.6,.7) -- (12.0,1.15) -- (11.4,1.05) -- cycle;
\filldraw[color = black, fill = red!9, thick] (11.9,1.9) -- (12.0,1.15) -- (12.57,0.65) -- cycle;
\filldraw[color = black, fill = red!9, thick] (11.9,1.9) -- (10.44,.25) -- (11.4,1.05) -- cycle;
\filldraw[color = black, fill = red!9, thick] (12.57,0.65) -- (10.44,.25) -- (11.6,.7) -- cycle;

\filldraw[draw=none, fill = orange!15, thick]
  (14.6,1.05) -- (14.0,1.15) -- (14.1,1.9) -- (14.6,1.05) -- (14.43,1.17) -- (14.16,1.65) -- (14.12,1.23) -- (14.43,1.17) --
  cycle;

\draw[] (14.16,1.65) -- (14.12,1.23) -- (14.43,1.17) --
  cycle;

\filldraw[color = black, fill = red!9, thick] (14.4,.7) -- (14.0,1.15) -- (14.6,1.05) -- cycle;
\filldraw[color = black, fill = red!9, thick] (14.1,1.9) -- (14.0,1.15) -- (13.43,0.65) -- cycle;
\filldraw[color = black, fill = red!9, thick] (14.1,1.9) -- (15.56,.25) -- (14.6,1.05) -- cycle;
\filldraw[color = black, fill = red!9, thick] (13.43,0.65) -- (15.56,.25) -- (14.4,.7) -- cycle;
\filldraw[color = black, fill = red!9, thick] (13.43,0.65) -- (14.0,1.15) -- (14.4,.7) -- cycle;

\filldraw[color = black, fill = red!9, thick] (14.6,1.05) -- (15.56,.25) -- (14.4,.7) -- cycle;


\filldraw[black] (10,0) circle (2pt);
\filldraw[black] (16,0) circle (2pt);
\filldraw[black] (13,\x) circle (2pt);

\filldraw[black] (11.93,2.35) circle (2pt);
\filldraw[black] (14.07,2.35) circle (2pt);
\filldraw[black] (13,.5) circle (2pt);

\filldraw[black] (13,2.8) circle (2pt);
\filldraw[black] (12.7,3.4) circle (2pt);
\filldraw[black] (13.3,3.4) circle (2pt);

\filldraw[black] (12.0,1.15) circle (2pt);
\filldraw[black] (11.6,.7) circle (2pt);
\filldraw[black] (11.4,1.05) circle (2pt);

\filldraw[black] (14.0,1.15) circle (2pt);
\filldraw[black] (14.4,.7) circle (2pt);
\filldraw[black] (14.6,1.05) circle (2pt);

\filldraw[black] (10.44,0.25) circle (2pt);
\filldraw[black] (11.9,1.9) circle (2pt);
\filldraw[black] (12.57,0.65) circle (2pt);

\filldraw[black] (15.56,0.25) circle (2pt);
\filldraw[black] (14.1,1.9) circle (2pt);
\filldraw[black] (13.43,0.65) circle (2pt);

\filldraw[black] (13,\x-0.5) circle (2pt);
\filldraw[black] (12.28,2.6) circle (2pt);
\filldraw[black] (13.72,2.6) circle (2pt);

\node[] at (14.24,1.33) {\footnotesize{$S_4$}};
\node[] at (13,3.75) {\footnotesize{$S_1$}};
\node[] at (12.68,2.95) {\footnotesize{$S_2$}};
\node[] at (13.32,2.95) {\footnotesize{$S_3$}};
\node[] at (10.7,2.6) {{$u$}};

\filldraw[black] (14.16,1.65) circle (1pt);
\filldraw[black] (14.12,1.23) circle (1pt);
\filldraw[black] (14.43,1.17) circle (1pt);

\filldraw[black] (13, \x-.76) circle (1pt);
\filldraw[black] (13.16,3.5)circle (1pt);
\filldraw[black] (12.84,3.5) circle (1pt);
\filldraw[black] (12.86,2.86) circle (1pt);
\filldraw[black] (12.69,3.22)circle (1pt);
\filldraw[black] (12.43,2.73) circle (1pt);
\filldraw[black] (13.14,2.86) circle (1pt);
\filldraw[black] (13.31,3.22) circle (1pt);
\filldraw[black] (13.57,2.73) circle (1pt);

\end{tikzpicture}